%% file: ifacconf.tex
\documentclass{ifacconf}
\usepackage{comment}
\usepackage{subcaption}
\usepackage{graphicx}      
\usepackage{natbib}        
\let\theoremstyle\relax
\usepackage{amsmath,amssymb,amsfonts,amsthm}
\usepackage{xcolor}
\usepackage{xurl}
\usepackage{bm, booktabs}
\usepackage{mathrsfs}
\usepackage{enumitem}
\input{PSmacros}

\begin{document}
\begin{frontmatter}

\title{Model-Free Optimization and Control of Rigid Body Dynamics: An Extremum Seeking for Vibrational Stabilization Approach}


\author[First]{Rohan Palanikumar} 
\author[First]{Ahmed A. Elgohary} 
\author[First]{Simone Martini}
\author[First]{Sameh A. Eisa}

\address[First]{Department of Aerospace Engineering and Engineering Mechanics,
University of Cincinnati, Cincinnati, OH 45221 USA (palanire@mail.uc.edu, elgohaam@mail.uc.edu, marti6so@ucmail.uc.edu, eisash@ucmail.uc.edu).}

\begin{abstract}
In this paper, we introduce a model-free, real-time, dynamic optimization and control method for a class of rigid body dynamics. Our method is based on a recent extremum seeking control for vibrational stabilization (ESC-VS) approach that is applicable to a class of second-order mechanical systems. The new ESC-VS method is able to stabilize a rigid body dynamic system about the optimal state of an objective function that can be unknown expression-wise, but assessable through measurements; the ESC-VS is operable by using \textit{only} one perturbation/vibrational signal. 
We demonstrate the effectiveness and the applicability of our ESC-VS approach via three rigid-body systems: (1) satellite attitude dynamics, (2) quadcopter attitude dynamics, and (3) acceleration-controlled unicycle dynamics. The results, including simulations with and without measurement delays/noise, illustrate the ability of our ESC-VS to operate successfully as a new methodology of optimization and control for rigid body dynamics. 
\end{abstract}


\end{frontmatter}
\section{Introduction}
Many systems across science and engineering fields, such as physics, bio-mechanics, and mechanical/aerospace/robotic engineering, are modeled and described using rigid body dynamics representation that usually arises from Newton's second law and/or Euler-Lagrange formulations (\cite{ardema2005analytical,ginsberg2008engineering,chaturvedi2011rigid}). That is, rigid body dynamic systems are typically studied and analyzed as being some classes of second-order mechanical systems. For decades, researchers have tried to propose and design control techniques and methods that can steer or optimize said mechanical systems, including rigid body dynamics, to achieve a desired objective. This includes many traditional optimization and control techniques, such as, but not limited to PID, MPC, LQR, geometric methods, among others; see, for example, \cite{posa2013direct,chaturvedi2011rigid,choi2004pid,moreno2018motion,golestani2022simple,silva2006lqr,bullo2019geometric}.

In relatively more recent literature, a new model-free optimization and control approach emerged for some classes of mechanical systems, namely the extremum seeking control (ESC) approach (\cite{suttner2022extremum,grushkovskaya2021extremum,eisa2023analyzing,suttner2024attitude}). ESC systems are feedback-based adaptive controllers that drive a dynamical system to an extremum (i.e., minimum or maximum) of an objective function. A primary advantage of these ESC systems is the ability to operate in real-time and model-free fashion when the mathematical model of the objective function is unknown and only measurements are available. ESC systems operate by employing periodic perturbation signals on the system, then use feedback laws involving the measured objective function to adapt the control input and steer the system to the extremum autonomously. For a review on many ESC techniques and analysis, see \cite{scheinker2024100,pokhrel2023higher}. Roughly speaking, ESC methods can be divided between two groups based on the nature of the perturbation signals(s). First group are those that require small amplitude, high frequency perturbation signals, such as classic ESC methods/techniques (e.g., \cite{ariyur2003real,KRSTICMain,guay2015time}). Second group, which is the most relevant to this paper, are those that require high amplitude, high frequency perturbation signals and applicable to mechanical systems, including rigid body dynamics (e.g., \cite{suttner2022extremum,grushkovskaya2021extremum,suttner2024attitude}).    

\textbf{Motivation and Contributions.} Inspired by flapping insects and hummingbirds mechanics, which achieve stable hovering using \textit{only} one high amplitude, high frequency perturbation signal, that is the flapping action of the wings itself, a new ESC approach was developed for that particular bio-physical phenomenon (\cite{enatural_hovering2024}). Said new ESC approach was then generalized, analyzed and introduced effectively to a class of second-order mechanical systems (\cite{elgohary2025letters}) where the ESC itself acts as a feedback adjuster/controller of what otherwise will be open-loop vibrational stabilization; hence, the approach in \cite{elgohary2025letters} is called extremum seeking for controlled vibrational stabilization (ESC-VS). As discussed in \cite{elgohary2025letters}, the ESC-VS approach had many advantages over ESC on mechanical systems literature such as \cite{suttner2022extremum}. In particular, ESC-VS can admit generalized forces that are quadratic in velocities, which is broader than what is available in ESC on mechanical systems literature and enables it to include more complex forces such as drag and aerodynamics. Additionally, ESC-VS operates by using \textit{only} one perturbation/vibrational signal as opposed to at least two perturbation signals in ESC on mechanical systems literature, which suits better the nature of some systems that have one source of vibration or need to save vibrational energy. For all the aforementioned reasons, we aim at introducing a novel framework of ESC-VS for a class of rigid body dynamics. We provide the formulation of the ESC-VS for rigid body approach and analyze it. Then, we put it into test by applying it to three different important rigid body applications, namely (1) satellite attitude dynamics, (2) quadcopter
attitude dynamics, and (3) acceleration-controlled unicycle dynamics. Results, including simulations with and without measurement delays/noise, illustrate the ability of the ESC-VS approach for rigid body to stabilize in real-time said systems successfully about the extremum of an objective function using, for the first time in literature, one vibrational signal.     

\section{Background on Extremum
Seeking for Vibrational Stabilization (ESC-VS)}\label{sec:background}
We start by considering a general class of second-order mechanical systems of the following form (\cite{elgohary2025letters}):
\begin{equation}
    \ddot{\mathbf{q}} = \mathbf{f(q,\dot{q})} + \boldsymbol{u}
    \label{eqn:system_model}
\end{equation}
where $\mathbf{q}, \mathbf{\dot{q}} \in \mathbb{R}^n$ are the generalized coordinates and velocities, respectively, of the system. The term \( \mathbf{f}(\mathbf{q}, \dot{\mathbf{q}}) = [f_1, f_2, \ldots, f_n]^T \in \mathbb{R}^n \) represents the drift in the system and can include nonlinear and damping effects. The vector $\boldsymbol{u} \in \mathbb{R}^n$ represents the control input to the system. An extremum seeking control for vibrational stabilization (ESC-VS), inspired by the flapping-based stabilization behavior observed in insects, is employed to stabilize the system around a desired configuration. The ESC-VS loop is designed using a single oscillatory signal to drive the second-order mechanical system \eqref{eqn:system_model} to the optimum of an objective function $J(\mathbf{q})$. This ESC-VS scheme can be written in a first-order state-space form as follows:
\begin{equation}
\frac{d}{dt}
\underbrace{\begin{bmatrix}
\mathbf{q}_{n \times 1} \\
\dot{\mathbf{q}}_{n \times 1} \\
\hat{u}
\end{bmatrix}}_{{\mathbf{x}}}
=
\underbrace{\begin{bmatrix}
\dot{\mathbf{q}}_{n \times 1} \\
\mathbf{f}(\mathbf{q}, \dot{\mathbf{q}}) + \mathbf{C} \hat{u} \\
0
 \end{bmatrix}}_{\mathbf{Z}(\mathbf{x})}
+
\underbrace{\begin{bmatrix}
\mathbf{0}_{n \times 1} \\
\mathbf{A} \\
k J(\mathbf{q})
\end{bmatrix} \omega \cos(\omega t)}_{\mathbf{Y}(\mathbf{x}, t)},
\label{eq:proposed_system_dynamics}
\end{equation}
where \( \mathbf{x} \in \mathbb{R}^{2n+1} \) is the state vector that contains the generalized coordinates, the generalized velocity, and the control estimate variable \( \hat{u} \). The term \( J(\mathbf{q}) \) is the objective function to be minimized. The scalar \( k > 0 \) denotes a learning gain and \( \omega \) is the frequency of the perturbation signal. Vectors \( \mathbf{C} = [c_1, \ldots, c_n]^T \in \mathbb{R}^n \) and \( \mathbf{A} = [a_1, \ldots, a_n]^T \in \mathbb{R}^n \) are positive design constants that define how the control estimate, \( \hat{u} \), and the sinusoidal input couple into the control law. We define a closed domain \( \mathscr{C}_0 \subset \mathbb{R}^{2n+1} \) and an open neighborhood \( \mathscr{C} \supset \mathscr{C}_0 \), and assume the following conditions hold throughout this region:

\begin{itemize}
    \item \textbf{A1.} The vector fields \( \mathbf{Z}(\mathbf{x}) \) and \( \mathbf{Y}(\mathbf{x}, t) \), along with their derivatives up to third order, are continuous and bounded in \( \mathscr{C}_0 \).
    
    \item \textbf{A2.} The drift term \( \mathbf{f}(\mathbf{q}, \dot{\mathbf{q}}) \) is quadratic in velocity, meaning that its third derivative with respect to \( \dot{\mathbf{q}} \) vanishes.
    
    \item \textbf{A3.} The objective function \( J(\mathbf{q}) \) is smooth and admits an isolated minimizer. That is, there exists \( \mathbf{q}^* \in \mathscr{C}_0 \) such that \( \nabla J(\mathbf{q}^*) = \mathbf{0} \), and \( \nabla J(\mathbf{q}) \neq \mathbf{0} \) for all \( \mathbf{q} \in \mathscr{C}_0 \setminus \{\mathbf{q}^*\} \). Furthermore, there exist class-\(\mathcal{K}\) functions \( \alpha_1 \) and \( \alpha_2 \) such that
    \[
    0 < \alpha_1(\|\mathbf{q} - \mathbf{q}^*\|) \leq J(\mathbf{q}) - J(\mathbf{q}^*) \leq \alpha_2(\|\mathbf{q} - \mathbf{q}^*\|).
    \]

    \item \textbf{A4.} The perturbation vector field \( \mathbf{Y}(\mathbf{x}, t) \) is time-periodic with period \( T = 2\pi/\omega \), and has zero mean over one period, i.e.,
    \[
    \int_{0}^{T} \mathbf{Y}(\mathbf{x}, t)\, dt = \mathbf{0}.
    \]
\end{itemize}

\begin{remark}
     Assumptions A1 and A4 are common in ESC and vibrational stabilization theory, ensuring regularity and periodicity of the system dynamics. Assumption A2 reflects the structural property that many physical systems exhibit forces that are at most quadratic in velocity, capturing effects such as aerodynamic drag or mechanical damping. Assumption A3 is standard in ESC literature and encodes convexity or convexity-like behavior near the optimum, ensuring a well defined minimization objective (\cite{VectorFieldGRUSHKOVSKAYA2018,grushkovskaya2021extremum}).
     \end{remark}

The ESC-VS system described in~\eqref{eq:proposed_system_dynamics} belongs to a class of mechanical systems that can be analyzed using the Variation-of-Constant (VOC) averaging approach, as seen in~\cite{bullo2002averaging,Maggia2020higherOrderAvg}. In our recent work~\cite{elgohary2025letters}, this framework was used to decompose the time-varying dynamics of the ESC-VS system into fast and slow subsystems. Through Lie bracket expansions and averaging techniques, we derived an equivalent averaged system that accurately captures the long-term behavior of the original ESC-VS dynamics \eqref{eq:proposed_system_dynamics}. In particular, we provided the following VOC-averaged system: 
\begin{equation}
    \dot{\bar{\mathbf{x}}} = \underbrace{\begin{bmatrix}
        \dot{\bar{\mathbf{q}}}_{n\times1} \\ \mathbf{f}(\bar{\mathbf{q}}, \dot{\bar{\mathbf{q}}}) + \mathbf{C}\hat{\bar{u}} \\ 0
    \end{bmatrix}}_{\mathbf{Z}(\bar{\mathbf{x}})}
    + \frac{1}{4} \begin{bmatrix}
        \mathbf{0}_{n\times1} \\ [\mathbf{M}_{22}]_{n\times n} [\mathbf{A}]_{n\times1} \\ -2k[\nabla_{\bar{\mathbf{q}}}J(\bar{\mathbf{q}})]_{1\times n} [\mathbf{A}]_{n\times1}
    \end{bmatrix},
    \label{eq:voc_averaged_system}
\end{equation}
where \(\mathbf{M}_{22}\) is an $n \times n$ matrix representing second-order mixed derivatives such that for any row index $i=1,...,n$ and column index $k=1,...,n$ of $\mathbf{M}_{22}$, we have $\mathbf{M}_{22}|_{(i,k)}=\sum_{j=1}^{n} \frac{\partial^2 f_i}{\partial \dot{q}_k \partial \dot{q}_j} a_j$, and \( \nabla_{\mathbf{q}} J(\mathbf{q}) \) denotes the gradient of \( J(\mathbf{q}) \) such that \( \nabla_{\mathbf{q}} J(\mathbf{q}) =
\begin{bmatrix} 
\frac{\partial J}{\partial q_1} & \frac{\partial J}{\partial q_2} & \dots & \frac{\partial J}{\partial q_n} 
\end{bmatrix}_{1 \times n} \).
\begin{remark}
    The trajectories of the ESC-VS system \eqref{eq:proposed_system_dynamics} remains \( \mathcal{O}(1/\omega) \)-close to the averaged system \eqref{eq:voc_averaged_system} for sufficiently high perturbation frequency \( \omega \) as in \cite[Corollary 1]{elgohary2025letters}.
\end{remark}
\begin{corollary}
    \cite[Corollary 1]{elgohary2025letters} If $\bar{\mathbf{x}}(t) \in \mathscr{C}_0, \forall t \in [0,\ \omega t_{f}]$ with $t_f > 0$ and $\bar{\mathbf{x}}(0)=\mathbf{x}(0)$, we have $|\mathbf{x}(t)-\bar{\mathbf{x}}(t)|=O(1/\omega)$ for $t \in [0,\ \omega t_{f}]$.
\end{corollary}

From a practical perspective, several important design insights also follow from the VOC-averaging analysis in~\cite{elgohary2025letters}. First, the excitation frequency \( \omega \) must be selected sufficiently large to ensure time-scale separation. Second, the learning rate \( k \) controls the responsiveness of the control adaptation and must be tuned based on the sensitivity of the objective function \( J(\mathbf{q}) \). Third, the feedback gain vector \( \bm{C} \) should be chosen larger than the oscillatory gain vector \( \bm{A} \) to ensure that the adaptive control estimate dominates the vibrational input. 

\section{Main Results: ESC-VS Rigid Body System}\label{sec:main}
We present the primary contribution of this paper, which is the extension of the new ESC-VS approach in \cite{elgohary2025letters} to a general class of rigid body systems. We can consider the rotational and translational rigid body dynamics to be \eqref{eq:rigid_body_rot_dyn} and \eqref{eq:rigid_body_tran_dyn}, respectively:

\begin{align}
    \mathbb{I}\bm{\dot{\Omega}} &= -(\bm{\Omega} \times \mathbb{I}\bm{\Omega}) + \bm{\tau}
    \label{eq:rigid_body_rot_dyn}\\
    \mathbb{M}\bm{\dot{v}} &= -(\bm{\Omega} \times \bm{v})+\bm{F}
    \label{eq:rigid_body_tran_dyn}
\end{align}

where $\bm{\dot{\Omega}}, \bm{\Omega} \in \mathbb{R}^3$ are the body-frame angular acceleration and velocity of the rigid body, respectively, $\mathbb{I} = diag(\begin{bmatrix}
    I_{xx} &I_{yy} &I_{zz}
\end{bmatrix}) \in \mathbb{R}^{3\times 3}$ is the diagonal matrix of inertia, $\bm{\tau}\in \mathbb{R}^{3}$ is the vector of external torques, $\bm{\dot{v}}, \bm{v} \in \mathbb{R}^3$ are the body-frame translational acceleration and velocity of the rigid body, respectively, $\mathbb{M}=mI_{3}\in \mathbb{R}^{3\times 3}$ is the mass matrix with $I_{3}$ being the $3\times 3$ identity matrix, and $\bm{F}\in \mathbb{R}^{3}$ is the vector of external forces. Isolating the terms, $\bm{\dot{\Omega}}$ and $\bm{\dot{v}}$, produces effectively second order terms comparable to $\ddot{\mathbf{q}}$. Furthermore, we impose the following assumption:
\begin{itemize}
\item \textbf{A5.} We assume that all inertial matrices $\mathbb{I}$ and $\mathbb{M}$ are constant and diagonal, meaning that all products of inertia terms are considered to be zero.
\end{itemize}
\begin{lemma}\label{lemma}
    If assumptions \textbf{A1}-\textbf{A5} are satisfied, then the rigid body dynamical system \eqref{eq:rigid_body_rot_dyn}-\eqref{eq:rigid_body_tran_dyn} is a member of the class of second-order mechanical systems \eqref{eqn:system_model}.
\end{lemma}
\begin{proof}
Consider the control input based on \cite{elgohary2025letters}:
\begin{equation}\label{eq:control}
    \bm{u} = \bm{C}{\hat{u}} + \bm{A}\omega\cos{(\omega t)},
\end{equation}
with vector of control matrices
\begin{equation}
    \bm{C} = \begin{bmatrix} c_1 & c_2 & c_3 & c_4 & c_5 & c_6 \end{bmatrix}^{\top},
    \bm{A} = \begin{bmatrix} a_1 & a_2 & a_3 & a_4 & a_5 & a_6 \end{bmatrix}^{\top}
    \label{eq:control_law}
\end{equation}
where
\begin{equation}
    c_i = \frac{c'_{i}}{b_i},\quad a_i = \frac{a'_{i}}{b_i}, \quad \forall i\in[1,6]\cap \mathbb{Z}^+,
\end{equation}
with $c'_{i}$ and $a'_{i}$ are control parameters and $b_i$ being the nonzero constant diagonal term of the matrix $B = diag(\mathbb{M}, \; \mathbb{I}) = diag(\begin{bmatrix}
    {b_1} &{b_2}  &{b_3}  &{b_4}  &{b_5}  &{b_6}
\end{bmatrix}) \in \mathbb{R}^{6\times 6}$ which groups the system's inertial components. Defining  the input vector, $\bm{u} = \bm{B}^{-1}\bm{u}' \in \mathbb{R}^{6 \times 1}$, with $u' = \begin{bmatrix}
    \bm{\tau} & \bm{F}
\end{bmatrix}^{\top}$ being the physical input vector, system \eqref{eq:rigid_body_rot_dyn} and \eqref{eq:rigid_body_tran_dyn}  can be rewritten as: 
\begin{equation}
    \underbrace{\begin{bmatrix}
        \bm{\dot{\Omega}}\\\bm{\dot{v}}
    \end{bmatrix}}_{\ddot{\mathbf{q}}} = \underbrace{\begin{bmatrix}
        -\mathbb{I}^{-1}(\bm{\Omega} \times \mathbb{I}\bm{\Omega})\\-\mathbb{M}^{-1}(\bm{\Omega} \times \bm{v})
    \end{bmatrix}}_{\mathbf{f(q,\dot{q})}} + \underbrace{\underbrace{\begin{bmatrix}
        \mathbb{I}^{-1} & \bm{0}_{3}\\
        \bm{0}_{3}&\mathbb{M}^{-1}
    \end{bmatrix}}_{\bm{B}^{-1}}\underbrace{\begin{bmatrix}
        \bm{\tau}\\\bm{F}
    \end{bmatrix}}_{\boldsymbol{u}'} }_{\bm{u}}.
    \label{eq:rigid_body_decomp}
\end{equation}
The rigid body dynamics \eqref{eq:rigid_body_decomp} is clearly in the form \eqref{eqn:system_model}.    
\end{proof}
\begin{remark}\label{remark2}
    We note that the constant inertial terms, $\bm{B}$, are lumped in the control matrices to be tuned, $\bm{C}$ and $\bm{A}$. This means that the ESC control law \eqref{eq:control} is still model-free as no need for accessing the information of $\bm{B}$ in practice.
\end{remark}
\begin{corollary}\label{corollary}
    Supposing 
    \begin{equation}\label{eq:law_hat}
        \dot{\hat{u}}=k J(\mathbf{q})\omega\cos{(\omega t)},
    \end{equation} applying the control input \eqref{eq:control} to the rigid body dynamics \eqref{eq:rigid_body_rot_dyn}-\eqref{eq:rigid_body_tran_dyn} (equivalent, per Lemma \ref{lemma}, to \eqref{eq:rigid_body_decomp}), leads to the ESC-VS rigid body dynamics being in the form \eqref{eq:proposed_system_dynamics} and its VOC-averaged system being in the form of \eqref{eq:voc_averaged_system}.
\end{corollary}
\begin{lemma}\label{lemma2}
    If assumptions \textbf{A1}-\textbf{A5} are satisfied and Corollary \ref{corollary} is in place, suppose $\bar{\textbf{x}}^* =\begin{bmatrix}
    \bar{\textbf{q}}^*& \bm{0} & \hat{\bar{u}}^*
\end{bmatrix}^{\top} \in \mathscr{C}_0,$ is an equilibrium point for the VOC-averaged system \eqref{eq:voc_averaged_system}, then $\bar{\textbf{x}}^*$ is an asymptotically locally uniformly stable equilibrium point.
\end{lemma}
\begin{proof}
    Per \cite[Proposition II.5]{elgohary2025letters}, $V = J(\mathbf{q}) - J(\mathbf{q}^*)$ is a viable Lyapunov function for the VOC-averaged system \eqref{eq:voc_averaged_system} with the equilibrium $\bar{\textbf{x}}^* =\begin{bmatrix}
    \bar{\textbf{q}}^*& \bm{0} & \hat{\bar{u}}^*
\end{bmatrix}^{\top}$. Thus, $\bar{\textbf{x}}^*\in \mathscr{C}_0$ is asymptotically locally uniformly stable for the system \eqref{eq:voc_averaged_system}.
\end{proof}
\begin{theorem}\label{theorem}
    If assumptions \textbf{A1}-\textbf{A5} are satisfied, and Lemma 3.1 and 3.4 in place, then the ESC-VS rigid body system \eqref{eq:rigid_body_decomp} is  
practically uniformly stable for 
$\mathscr{C}_0$.
\end{theorem}
\begin{proof}
     From Lemma 3.4, the equilibrium point $\bar{\textbf{x}}^* =\begin{bmatrix}
    \bar{\textbf{q}}^*& \bm{0} & \hat{\bar{u}}^*
\end{bmatrix}^{\top} \in \mathscr{C}_0$ is asymptotically locally uniformly stable for the system \eqref{eq:voc_averaged_system}. Thus, 
per \cite[Theorem II.3]{elgohary2025letters}, the ESC-VS rigid body system \eqref{eq:rigid_body_decomp} is practically uniformly stable for $\mathscr{C}_0$.
\end{proof}
\begin{remark}\label{remark_extend}
    Theorem \ref{theorem} effectively extends the results of \cite{elgohary2025letters} to the class of rigid body dynamics we consider in this paper. This ESC-VS rigid body design is depicted in Figure \ref{fig:esc_diagram}.
\end{remark}

In practice, an optional high-pass filter can be incorporated into the measurement feedback within the adaptation law to reduce the effect of low-frequency disturbances or bias in the objective function \cite{ECC2024,elgohary2025model}. The filter state \( h \) evolves according to:
\begin{equation}
    \dot{h} = -e h + J,
\end{equation}
where \( e > 0 \) is the filter gain. With this filtering in place, the adaptation law in \eqref{eq:control} for the estimated control input is adjusted to:
\begin{equation}
    \dot{\hat{u}} = k ( J-eh)  \omega \cos(\omega t).
    \label{eq:adaptive_control_estimate_hpf}
\end{equation}

\begin{figure*}[ht]
\centering
    \includegraphics[width=0.8\linewidth]{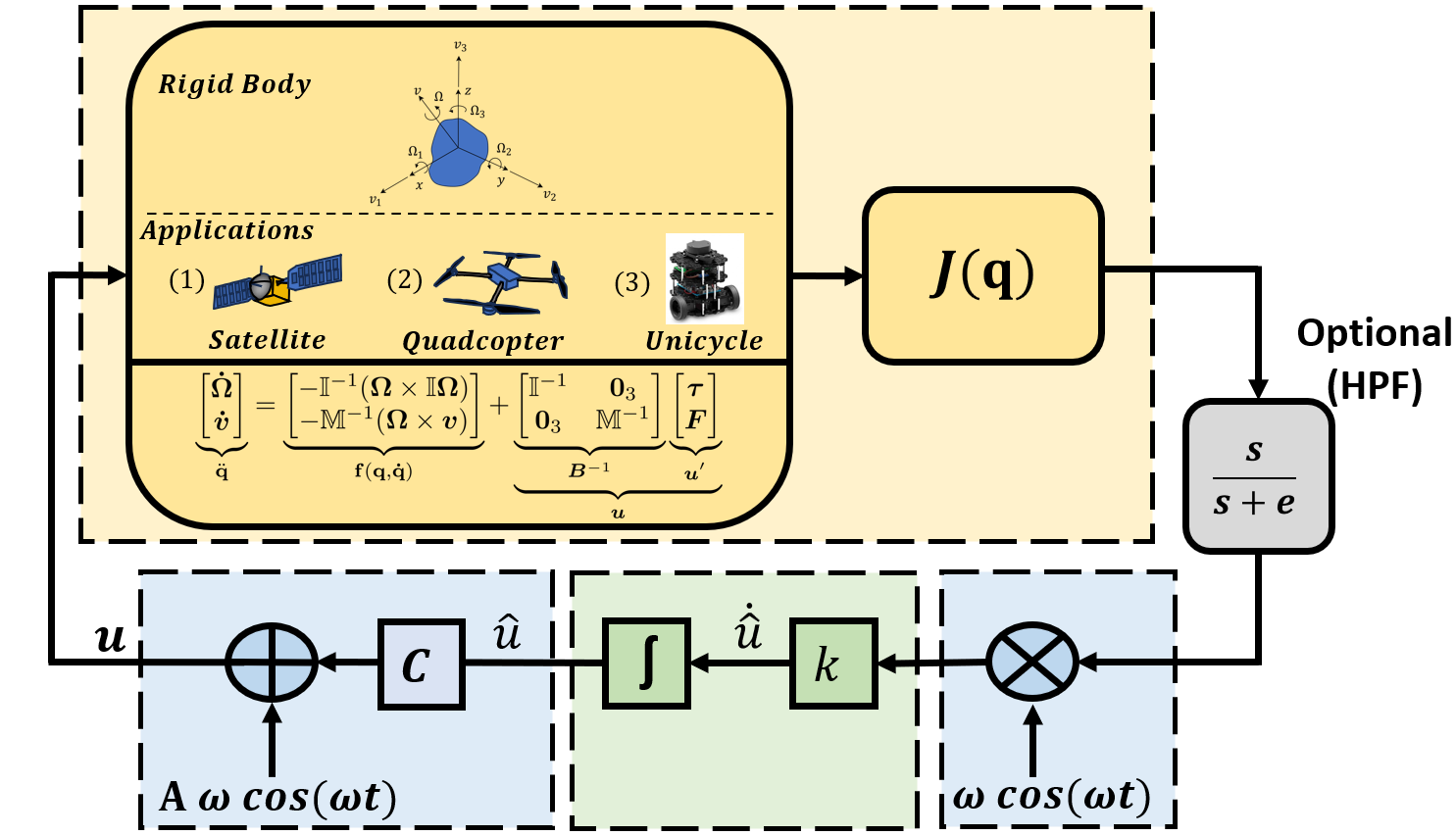}
    \caption{The proposed ESC-VS rigid body structure, which we demonstrate its effectiveness by three major applications.}
    \label{fig:esc_diagram}
\end{figure*}

In the following subsections, we will present three examples to show how our general result translates to the analysis of practical real-world rigid body systems.

\textbf{Control parameters tuning guidelines.}As stated in Remark \ref{remark_extend}, Theorem \ref{theorem} guarantees that the proposed ESC-VS system maintains the same stability properties of the ESC system in \cite[Theorem II.3]{elgohary2025letters}; hence, the same guidelines for parameter tuning mentioned in \cite{elgohary2025letters} applies to the proposed ESC-VS. That is, to ensure accurate applicability of the averaging results (see Corollary \ref{corollary}), $\omega$ must be set to a sufficiently large value. For $k$, which from the control estimate adaptation law \eqref{eq:law_hat} and the VOC-averaged system \eqref{eq:voc_averaged_system} represents the learning rate for $\nabla_{\textbf{q}} J(\textbf{q})$, it should be tuned from lower to higher rates to find the appropriate value suitable for the system's map. Moreover, to allow the ESC control action to prevail over the drift contribution and enable stabilization per \cite[Proposition II.5]{elgohary2025letters}, it is needed for $\bm{C}$ to be picked sufficiently larger than $\bm{A}$, while the latter can be set adequately small; This also can be seen directly from the second line in the VOC-averaged system \eqref{eq:voc_averaged_system} where larger $\bm{C}$ values mean the term $\mathbf{C}\hat{\bar{u}}$ is dominant over the drift forces and able to control $\dot{\bar{\mathbf{q}}}$. Finally, from \eqref{eq:voc_averaged_system}, one can, if necessary, select the elements of $\bm{A}$ to scale the components of $\nabla_{\textbf{q}} J(\textbf{q})$. 

\subsection{Application 1: Satellite Attitude Control}
We begin by considering a satellite with three reaction wheels that are aligned with the body frame axes of the satellite. The rotational dynamics for the satellite (\cite{wie1989quaternion}, \cite{suttner2022extremum}, \cite{suttner2024attitude}) is given in \eqref{eq:satellite_att_dynamics}. However, we include the dissipation term as in \cite{suttner2022extremum} and add the coupling effect between the satellite body and the reaction wheels:
\begin{equation}
    \mathbb{I}\bm{\dot{\Omega}} = -(\bm{\Omega} \times \mathbb{I}\bm{\Omega}) - (\bm{\Omega} \times \bm{\Omega}_{RW}\mathbb{I}_{RW}) - D\bm{\Omega} + \bm{\tau}_{RW}
    \label{eq:satellite_att_dynamics}
\end{equation}
with
\begin{equation}
    \bm{\tau}_{RW} = -\mathbb{I}_{RW}\bm{\dot \Omega}_{RW}
\end{equation}
where the terms $\bm{\dot{\Omega}}_{RW}, \bm{\Omega}_{RW} \in \mathbb{R}^3$ are the angular acceleration and velocity of the reaction wheels, $\mathbb{I}_{RW} = diag(\begin{bmatrix}
    I_{RW_{xx}} &I_{RW_{yy}} &I_{RW_{zz}}
\end{bmatrix}) \in \mathbb{R}^{3\times 3}$ represent the reaction wheel inertia matrix, and $D = diag(\begin{bmatrix}
    d_{{1}} &d_{{2}} &d_{{3}}
\end{bmatrix}) \in \mathbb{R}^{3\times 3}$ is the diagonal damping matrix that captures dissipative effects.
In this scenario, the translational rigid body dynamics \eqref{eq:rigid_body_tran_dyn} is not considered, however, \eqref{eq:satellite_att_dynamics} adds some additional terms to the general angular rigid body dynamics of \eqref{eq:rigid_body_rot_dyn}. Nevertheless, \eqref{eq:satellite_att_dynamics}  can still be rewritten to fit the general form of \eqref{eqn:system_model} as: 
\small
\begin{align}\label{satform1}
    \underbrace{\begin{bmatrix}
        \bm{\dot{\Omega}}\\\dot{\bm{\Omega}}_{RW}
    \end{bmatrix}}_{\ddot{\mathbf{q}}} =& \underbrace{\begin{bmatrix}
        -\mathbb{I}^{-1}(\bm{\Omega} \times \mathbb{I}\bm{\Omega}) -\mathbb{I}^{-1}(\bm{\Omega} \times \bm{\Omega}_{RW}I_{RW}) -\mathbb{I}^{-1}D\bm{\Omega}\\\bm{0}_3
    \end{bmatrix}}_{\mathbf{f(q,\dot{q})}} + \nonumber\\ &+\underbrace{\begin{bmatrix}
        \mathbb{I}^{-1} \bm{\tau}_{RW}\\-\mathbb{I}_{RW}^{-1} \bm{\tau}_{RW}
    \end{bmatrix}}_{\boldsymbol{u}}. 
\end{align}
\normalsize
Since $\mathbb{I}_{RW}$ is a constant diagonal inertia matrix, per Lemma \ref{lemma}, \eqref{satform1} is in the form \eqref{eqn:system_model} and Theorem 3.5 is applicable.

To steer the attitude of the satellite toward a desired configuration, we assume that the satellite has access to measurement of the objective function 
\begin{equation}
    J = \text{Q}_{e1}^2 + \text{Q}_{e2}^2 + \text{Q}_{e3}^2,
\end{equation}
where $\text{Q}_{e1}$, $\text{Q}_{e2}$, and $\text{Q}_{e3}$ are the vector components of the error quaternion between the desired and current attitude. We omit the scalar component of the error quaternion, $\text{Q}_{e4}$, from the objective function, as it will be equal to $1$ or $-1$ when the vector components are zero. The definition of the error quaternion follows from \cite{wie1989quaternion}:
\begin{equation}
    \begin{bmatrix}
        \text{Q}_{e1} \\ \text{Q}_{e2} \\ \text{Q}_{e3} \\ \text{Q}_{e4}
    \end{bmatrix}
    =
    \begin{bmatrix}
        \text{Q}_{4d} & \text{Q}_{3d} & -\text{Q}_{2d} & -\text{Q}_{1d} \\
        -\text{Q}_{3d} & \text{Q}_{4d} & \text{Q}_{1d} & -\text{Q}_{2d} \\
        \text{Q}_{2d} & -\text{Q}_{1d} & \text{Q}_{4d} & -\text{Q}_{3d} \\
        \text{Q}_{1d} & \text{Q}_{2d} & \text{Q}_{3d} & \text{Q}_{4d} \\
    \end{bmatrix}
    \begin{bmatrix}
        \text{Q}_1 \\ \text{Q}_2 \\ \text{Q}_3 \\ \text{Q}_4
    \end{bmatrix},
\end{equation}
where $\text{Q}_{1d}$, $\text{Q}_{2d}$, $\text{Q}_{3d}$, and $\text{Q}_{4d}$ are the desired quaternions, $\text{Q}_{e4}$ is the scalar quaternion error element, and $\textbf{Q} =\begin{bmatrix}
    \text{Q}_1 \ \text{Q}_2 \ \text{Q}_3 \ \text{Q}_4
\end{bmatrix}^T$ is the quaternion describing the current attitude. The quaternion attitude description is updated by the kinematic differential equation:

\begin{equation}\label{eq:kinquat}
    \dot{\textbf{Q}} = \frac{1}{2} \begin{bmatrix}
        \text{Q}_4 & -\text{Q}_3 & \text{Q}_2 \\
        \text{Q}_3 & \text{Q}_4 & -\text{Q}_1 \\
        -\text{Q}_2 & \text{Q}_1 & \text{Q}_4 \\
        -\text{Q}_1 & -\text{Q}_2 & -\text{Q}_3
    \end{bmatrix} \bm{\Omega}.
\end{equation}
Having selected the generalized coordinates such that $\dot{\textbf{q}} = \begin{bmatrix}
    \bm{\Omega} & \bm{\Omega}_{RW}
\end{bmatrix}^{\top}$, and assuming that $\textbf{Q}=\varphi(\textbf{q})$ exists and well defined for a smooth map $\varphi$, we use \eqref{eq:kinquat} to generate the measurements of $J$.


\subsection{Application 2: Quadcopter Attitude Control}
Given the quadcopter underactuated dynamics, generally the translational and attitude control loops are tackled separately. For proof of concept, we focus on applying the ESC-VS rigid body system to stabilize the attitude of the quadcopter. We will consider the quadcopter dynamics (\cite{quadcopterMPCspringer}, \cite{kai2017quadrotordrag}) without coupling between the rotors and the body of the quadcopter for simplicity, but with the inclusion of a linear dissipation term dependent on $\bm{\Omega}$
\begin{align}\label{eq:quadcopter_rotdyn}
    \mathbb{I}\dot{\bm{\Omega}} &= -(\bm{\Omega} \times \mathbb{I}\bm{\Omega})  - k_r\bm{\Omega} + \bm{\tau},\\
    \mathbb{M}\bm{\dot{v}} &= -(\bm{\Omega} \times \bm{v})- k_d\bm{v}+\bm{F} \label{eq:quadcopter_tradyn}
\end{align}
where $k_r= diag(\begin{bmatrix}
    k_{r1} &k_{r2} &k_{r3}
\end{bmatrix}) \in \mathbb{R}^{3\times 3}$ and $k_d= diag(\begin{bmatrix}
    k_{d1} &k_{d2} &k_{d3}
\end{bmatrix}) \in \mathbb{R}^{3\times 3}$  represent the aerodynamic drag matrices.  In the general quadcopter model, the only body frame applied force is the $z$-axis thrust force $T$, so that $\bm{F} = \begin{bmatrix}
     0&0&T
 \end{bmatrix}^{\top}$ For simplicity, in the following we will select a constant thrust input $T$ and apply the ESC-VS for the stabilization of the attitude dynamics only. To this end, \eqref{eq:quadcopter_rotdyn}  can be rewritten as
\begin{equation}\label{eq:quadcopterdyn}
    \underbrace{\begin{bmatrix}
        \bm{\dot{\Omega}}
    \end{bmatrix}}_{\ddot{\mathbf{q}}} = \underbrace{\begin{bmatrix}
        -\mathbb{I}^{-1}(\bm{\Omega} \times \mathbb{I}\bm{\Omega}- k_r\bm{\Omega})
    \end{bmatrix}}_{\mathbf{f(q,\dot{q})}} + \underbrace{
        \mathbb{I}^{-1}\bm{\tau}
    }_{\bm{u}}.
\end{equation}
Equation \eqref{eq:quadcopterdyn} clearly satisfy the form \eqref{eqn:system_model} per Lemma 3.1. As a result, Theorem 3.5 is applicable.

We select the objective function measurements as
\begin{equation}
    J = (\psi - \psi_d)^2 + (\theta - \theta_d)^2 + (\phi - \phi_d)^2,
\end{equation}
where $\bm{\eta} =\begin{bmatrix}
    \psi &\theta &\phi
\end{bmatrix}^{\top}$ is the vector of the 3-2-1 Euler Angles and $\psi_d$, $\theta_d$, and $\phi_d$ represent the respective  desired attitude. The propagation of the Euler Angle sequence is done using the kinematic differential equation as follows (\cite{schaub2018analytical}):
\begin{equation}\label{eq:kineuler}
    \begin{bmatrix}
        \dot{\psi} \\ \dot{\theta} \\ \dot{\phi}
    \end{bmatrix}
    = \frac{1}{\cos{(\theta)}}
    \begin{bmatrix}
        0 & \sin{(\phi)} & \cos{(\phi)} \\
        0 & \cos{(\phi)}\cos{(\theta)} & -\sin{(\phi)}\cos{(\theta)} \\
        \cos{(\theta)} & \sin{(\phi)}\sin{(\theta)} & \cos{(\phi)}\sin{(\theta)}
    \end{bmatrix} \bm{\Omega}
\end{equation}
Similar to the previous application, we use \eqref{eq:kineuler} to generate the measurements of $J$.
\subsection{Application 3: Acceleration-controlled Unicycle}
Lastly, we will consider the accelerating unicycle dynamics model with the inclusion of dissipative effects (\cite{suttner2019accelunicycle,suttnerkrisitic2023unicycletorque,elgohary2025model}). The equations of motion are a direct result of restricting the evolution of \eqref{eq:rigid_body_tran_dyn} on to a line and the evolution of \eqref{eq:rigid_body_rot_dyn} about the $z$-axis leading to:
\begin{equation}\label{eq:unicycle_ESC}
\begin{aligned}
    \underbrace{\begin{bmatrix}
        \dot{\Omega} \\ \dot{v}
    \end{bmatrix}}_{\ddot{\mathbf{q}}} &= \underbrace{\begin{bmatrix}
        -d_\Omega \Omega \\ -d_v v
    \end{bmatrix}}_{\mathbf{f}(\mathbf{q},\dot{\mathbf{q}})} + \underbrace{\begin{bmatrix}
        \tau \\ F
    \end{bmatrix}}_{\mathbf{u}},
\end{aligned}
\end{equation}
where $v$ is the linear velocity and $\Omega$ is the angular velocity. The terms $d_v$ and $d_{\Omega}$ represent dissipative effects. Clearly the system \eqref{eq:unicycle_ESC} is in the form of \eqref{eqn:system_model} per Lemma 3.1. As a result, Theorem 3.5 is applicable.
We assume that the unicycle system can constantly measure the objective function
\begin{equation}
    J = (x-x_d)^2 + (y-y_d)^2,
\end{equation}
where $\bm{p} =\begin{bmatrix}
    x & y 
\end{bmatrix}^{\top}$ is the current position vector and $x_d$ and $y_d$ are the respective desired coordinates.
The combined effect of $v$ and $\Omega$ cause a position displacement in the $x$ and $y$  coordinate axes according to the kinematic differential equation 
\begin{align}\label{eq:xdot}
    \dot{x} &= v\cos{(\vartheta)} \\
    \dot{y} &= v\sin{(\vartheta)} \label{eq:ydot} \\
    \dot{\vartheta} &= \Omega \label{eq:thetadot}
\end{align}
Similar to the previous application, we use \eqref{eq:xdot} and \eqref{eq:ydot} to generate the measurements of $J$.
\section{Numerical Simulations and Discussion}\label{sec:sim}
In this section, we demonstrate the ability of the ESC-VS rigid body structure (Figure \ref{fig:esc_diagram}) to stabilize both the satellite and quadcopter attitude, as well as to steer the unicycle model to the desired state. The ESC-VS rigid body system operates in real-time using only a single high-frequency perturbation signal and measurements of the objective function. The parameter values assigned to each simulation of the three presented applications are reported in Table \ref{tab:tabsat}, Table \ref{tab:tabquad}, and Table \ref{tab:tabuni}, respectively. While the model parameters of each platform are set according to respective literature (e.g., \cite{suttner2022extremum,quadcopter2007params,suttnerkrisitic2023unicycletorque,elgohary2025model,suttner2024attitude,suttner2019accelunicycle,kai2017quadrotordrag}), the control parameters are selected following the tuning guidelines highlighted at the end of Section \ref{sec:main}.

We observe in Figure~\ref{fig:satellite_simulation_results} that the vector components of the quaternion, $\text{Q}$, all oscillate about zero, and the scalar component of the quaternion converges to a value of $1$, due to the holonomic constraint (also known as the unit-norm constraint). Also, noting that the objective function reaches the minimum, we can state that the satellite has successfully stabilized about the desired attitude using the ESC-VS control loop. Furthermore, there is minimal change in the control estimate, satellite angular velocity, and reaction wheel angular velocity as time progresses. That is, the signals remain bounded and settle to small oscillations about their steady state values. For the quadcopter, from Figure~\ref{fig:quadcopter_simulation_results}, we observe that all Euler Angles oscillate about the desired zero-angle configuration, with the objective function reaching the minimum value. Hence, we can conclude that the quadcopter attitude has been stabilized successfully by the ESC-VS method.
The results for the accelerating unicycle model are presented in Figure~\ref{fig:unicycle_simulation_results}, from which we observe that the measurement function is minimized successfully leading to both $x$ and $y$ position variables to oscillate about the desired position coordinates $\begin{bmatrix}
    1 & 1
\end{bmatrix}^{\top} \ \text{m}$, demonstrating successful source-seeking using the ESC-VS structure. Finally, the obtained results highlight the potential impact of our main contribution of extending the ESC-VS to rigid body dynamics, opening its implementation to further numerous real-world applications. All code and simulation files are available at \cite{elgohary2025github}.

\begin{table}[h!]
\centering
\caption{ESC-VS Satellite System Parameters and Initial Conditions}\label{tab:tabsat}
\begin{tabular}{@{}ll@{}}
\toprule
\textbf{Parameter} & \textbf{Value} \\ 
\midrule
$I_{xx}, I_{yy}, I_{zz}$ & $1,\ 2,\ 3$ \\
$I_{RW_{xx}}, I_{RW_{yy}}, I_{RW_{zz}}$ & $0.005,\ 0.005,\ 0.005$ \\
$d_1, d_2, d_3$ & $0.2,\ 0.4,\ 0.6$ \\[3pt]
$\bm{Q}(0)$ & $\begin{bmatrix} 0.57 & 0.57 & 0.57 & 0.159 \end{bmatrix}^T$ \\
$\bm{Q}_d$ & $\begin{bmatrix} 0 & 0 & 0 & 1 \end{bmatrix}^T$ \\[3pt]
$\bm{\Omega}(0)$ & $\begin{bmatrix} 0.01 & 0.01 & 0.01 \end{bmatrix}^T\ \text{rad/s}$ \\
$\bm{\Omega}_{RW}(0)$ & $\begin{bmatrix} 0 & 0 & 0 \end{bmatrix}^T\ \text{rad/s}$ \\[3pt]
$\hat{u}(0)$ & $0$ \\
$h(0)$ & $0$ \\[3pt]
$a_1, a_2, a_3$ & $4.167\times10^{-5},\ 6.25\times10^{-5},\ 5.55\times10^{-5}$ \\
$a_4, a_5, a_6$ & $-8.3\times10^{-3},\ -2.5\times10^{-3},\ -3.3\times10^{-3}$ \\
$c_1, c_2, c_3$ & $3.375,\ 3.4167,\ 2.3056$ \\
$c_4, c_5, c_6$ & $-675,\ -1366.6667,\ -1383.3333$ \\
$k, e$ & $2.4,\ 15$ \\
$\omega$ & $36$\\
\bottomrule
\end{tabular}
\end{table}

\begin{figure*}[htbp]
\centering
    \begin{subfigure}[b]{0.33\linewidth}
    \includegraphics[width=\linewidth]{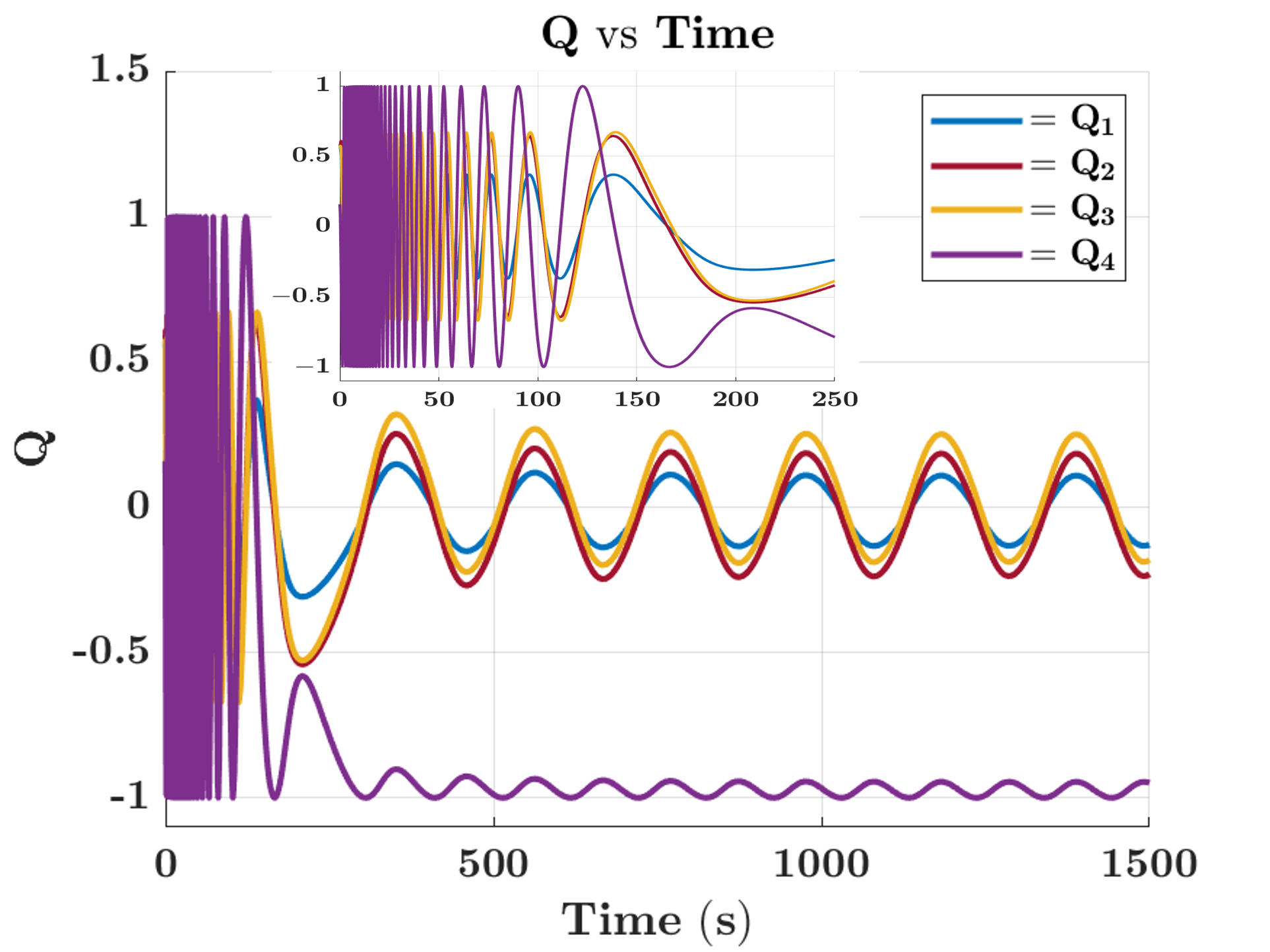}
    \caption{}
    \end{subfigure}\hfill
    \begin{subfigure}[b]{0.33\linewidth}
    \includegraphics[width=\linewidth]{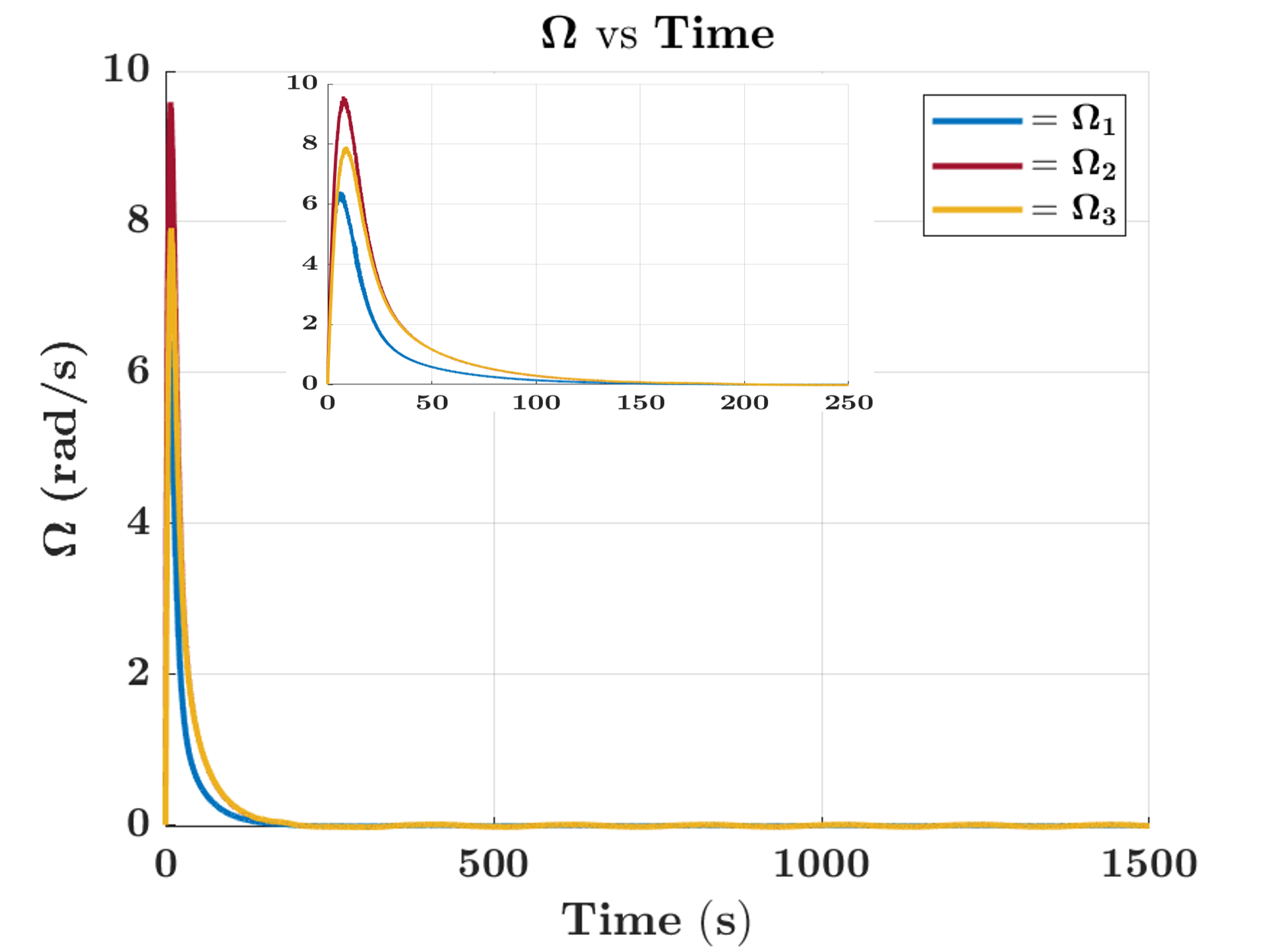}
    \caption{}
    \end{subfigure}\hfill
    \begin{subfigure}[b]{0.33\linewidth}
    \includegraphics[width=\linewidth]{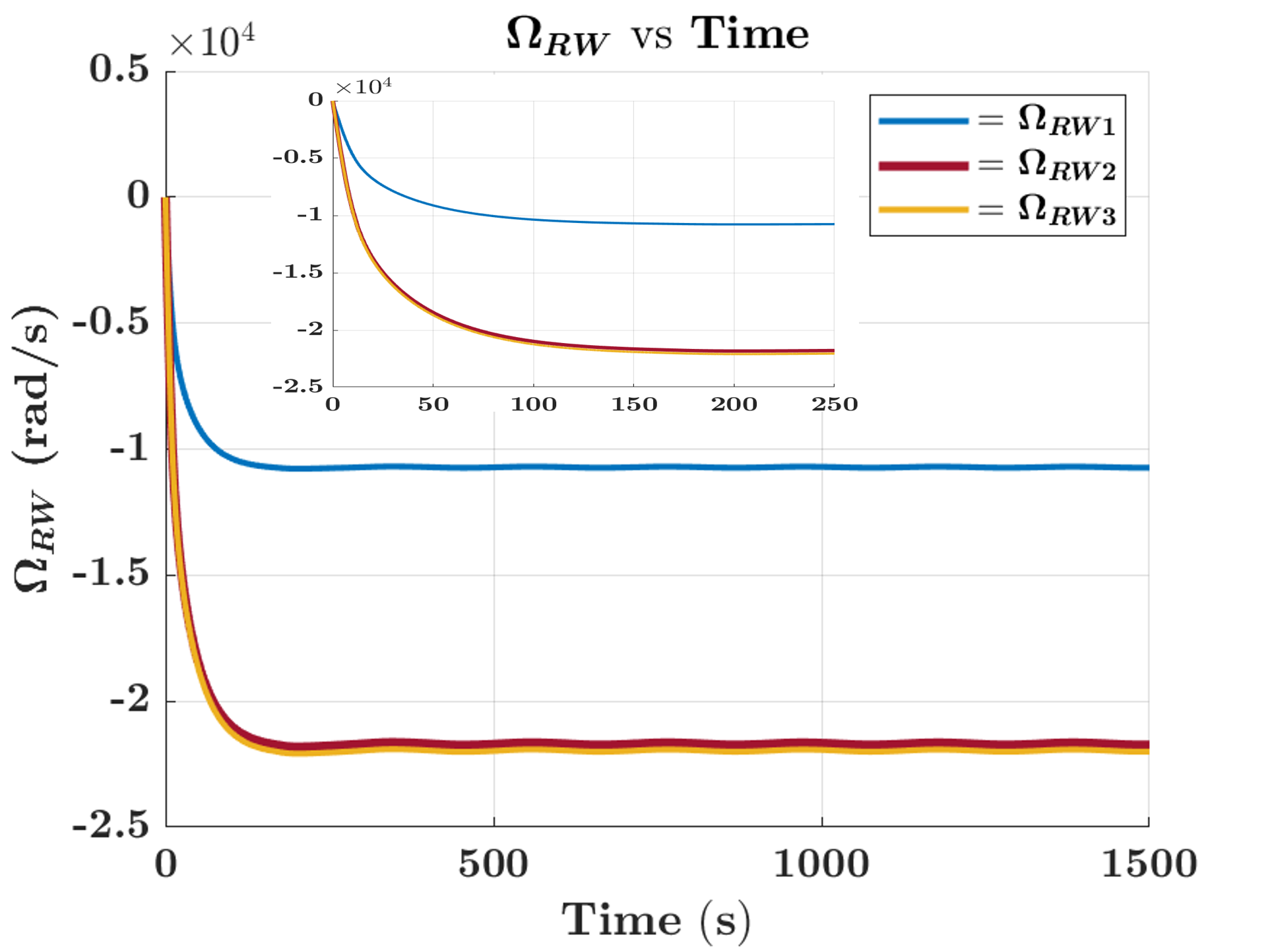}
    \caption{}
    \end{subfigure}
    \begin{subfigure}[b]{0.33\linewidth}
    \centering
    \includegraphics[width=\linewidth]{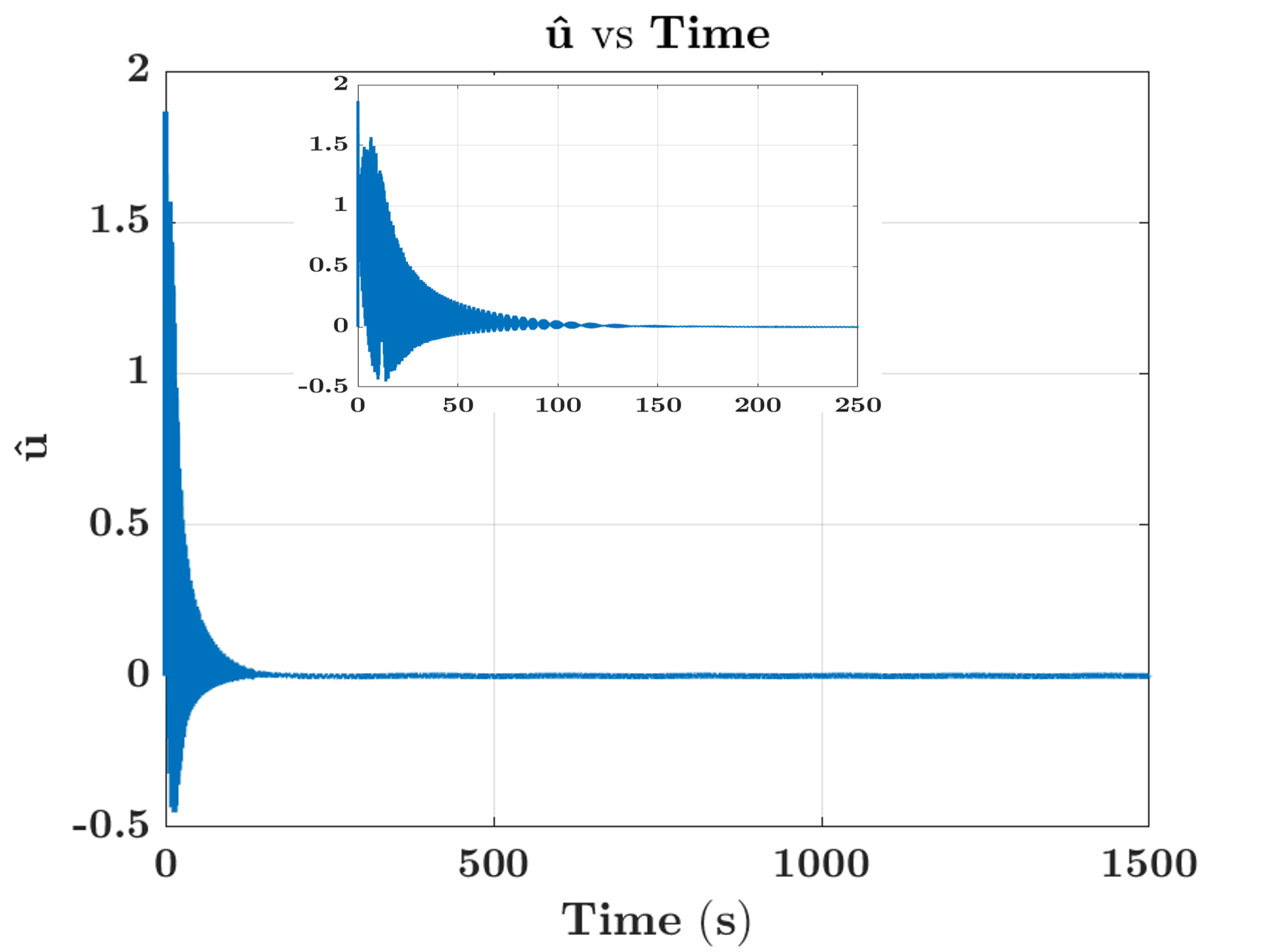}
    \caption{}
    \end{subfigure}\hspace{0.02\linewidth}
    \begin{subfigure}[b]{0.33\linewidth}
    \centering
    \includegraphics[width=\linewidth]{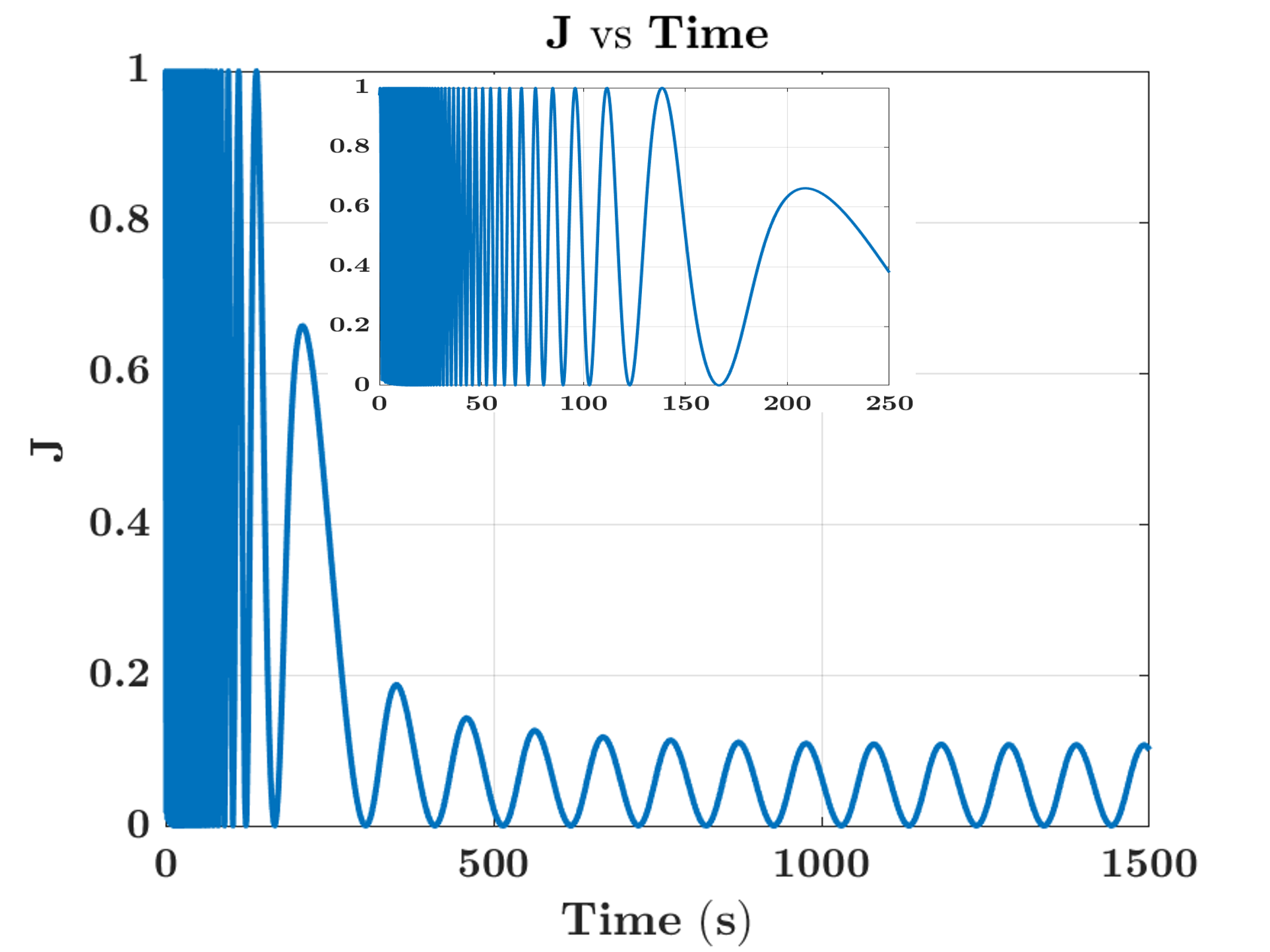}
    \caption{}
    \end{subfigure}
    \caption{Simulation Results for the satellite. (a) Quaternion, $\bm{Q}$, (b) Satellite angular velocity, $\bm{\Omega}$, (c) Reaction wheel angular velocity, $\bm{\Omega}_{RW}$, (d) Control estimate, $\hat{u}$, (e) Objective function measurements, $J$.}
    \label{fig:satellite_simulation_results}
\end{figure*}

\begin{table}[h!]
\centering
\caption{ESC-VS Quadcopter System Parameters and Initial Conditions}\label{tab:tabquad}
\begin{tabular}{@{}ll@{}}
\toprule
\textbf{Parameter} & \textbf{Value} \\ 
\midrule
$I_{xx}, I_{yy}, I_{zz}$ & $0.0075,\ 0.0075,\ 0.013$ \\
$k_{r1}, k_{r2}, k_{r3}$ & $0.1,\ 0.1,\ 0.15$ \\[3pt]
$\psi(0),\ \theta(0),\ \phi(0)$ & $0.1745,\ 0.2618,\ 0.2094\ \text{rad}$ \\
$\psi_d,\ \theta_d,\ \phi_d$ & $0,\ 0,\ 0\ \text{rad}$ \\[3pt]
$\bm{\Omega}(0)$ & $\begin{bmatrix} 0 & 0 & 0 \end{bmatrix}^T\ \text{rad/s}$ \\
$\hat{u}(0)$ & $0$ \\
$h(0)$ & $0$ \\[3pt]
$a_1, a_2, a_3$ & $6.6667\times10^{-4},\ 1\times10^{-3},\ 6.5385\times10^{-4}$ \\
$c_1, c_2, c_3$ & $20.5333,\ 31.6,\ 14.2308$ \\
$k, e$ & $3.95, \ 4$ \\
$\omega$ & $25$ \\
\bottomrule
\end{tabular}
\end{table}

\begin{figure*}[htbp]
    \centering
    \begin{subfigure}[b]{0.33\linewidth}
    \centering
    \includegraphics[width=\linewidth]{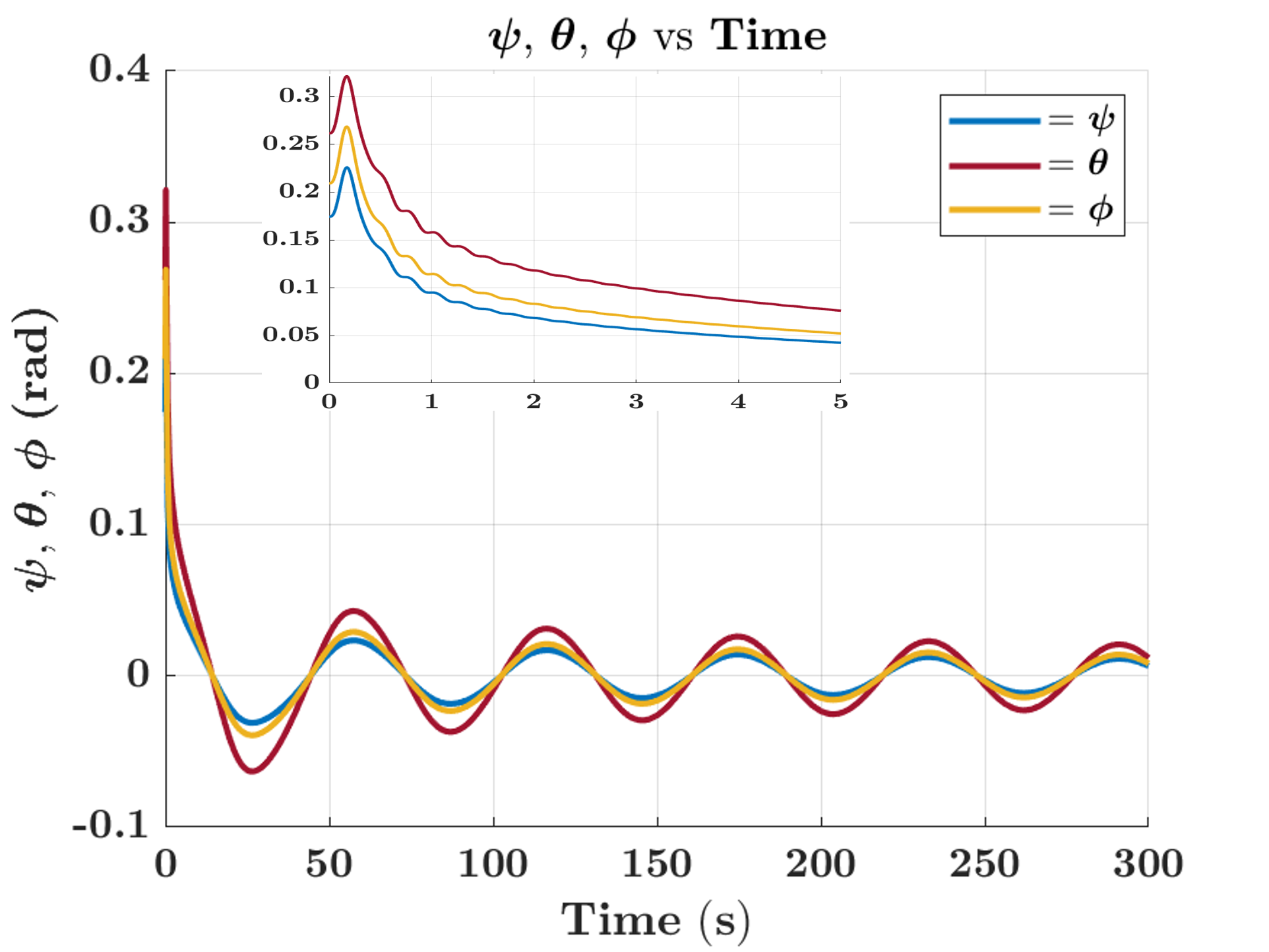}
    \caption{}
    \end{subfigure}\hspace{0.02\linewidth}
    \begin{subfigure}[b]{0.33\linewidth}
    \centering
    \includegraphics[width=\linewidth]{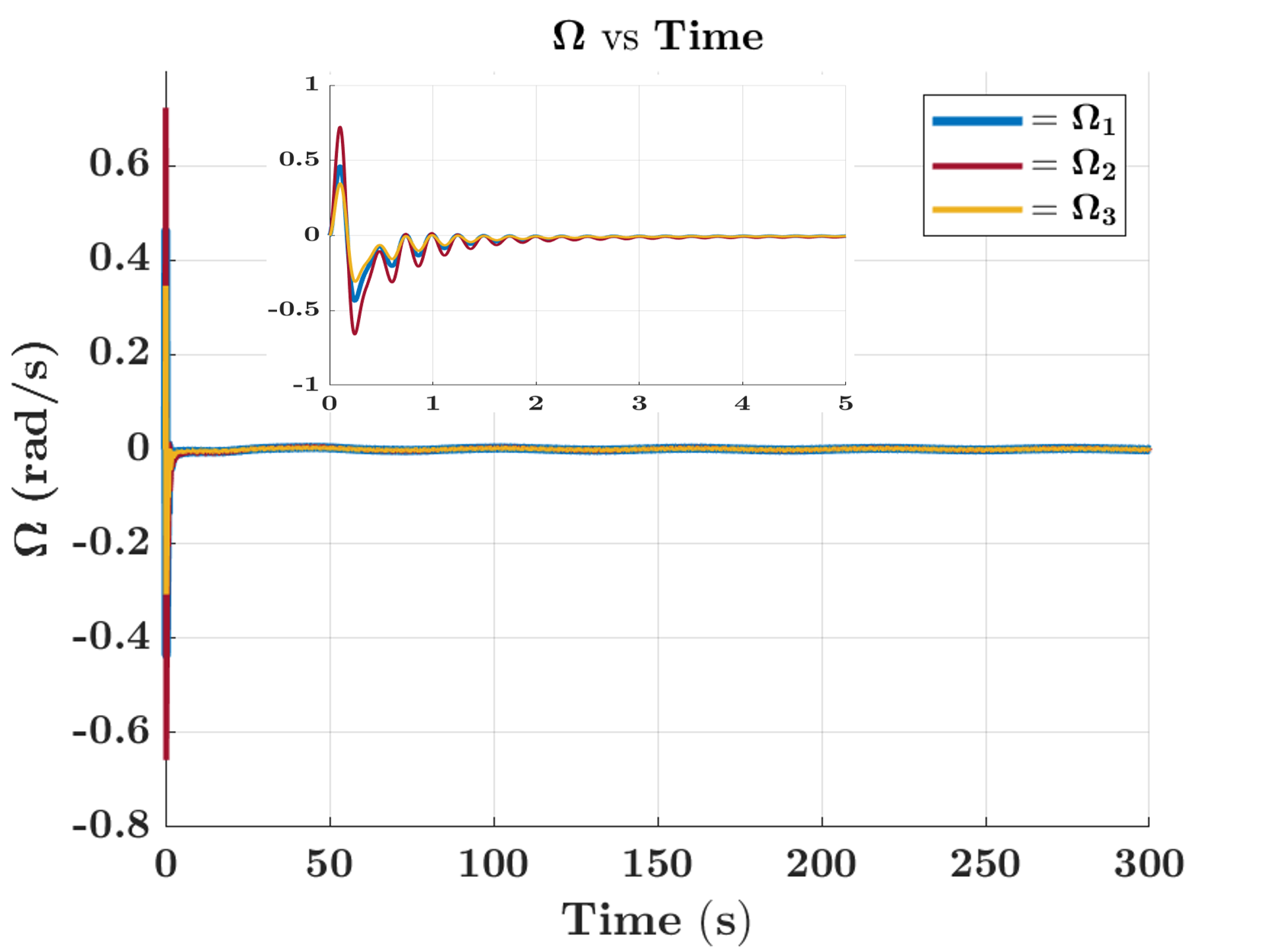}
    \caption{}
    \end{subfigure}\hfill
    \begin{subfigure}[b]{0.33\linewidth}
    \centering
    \includegraphics[width=\linewidth]{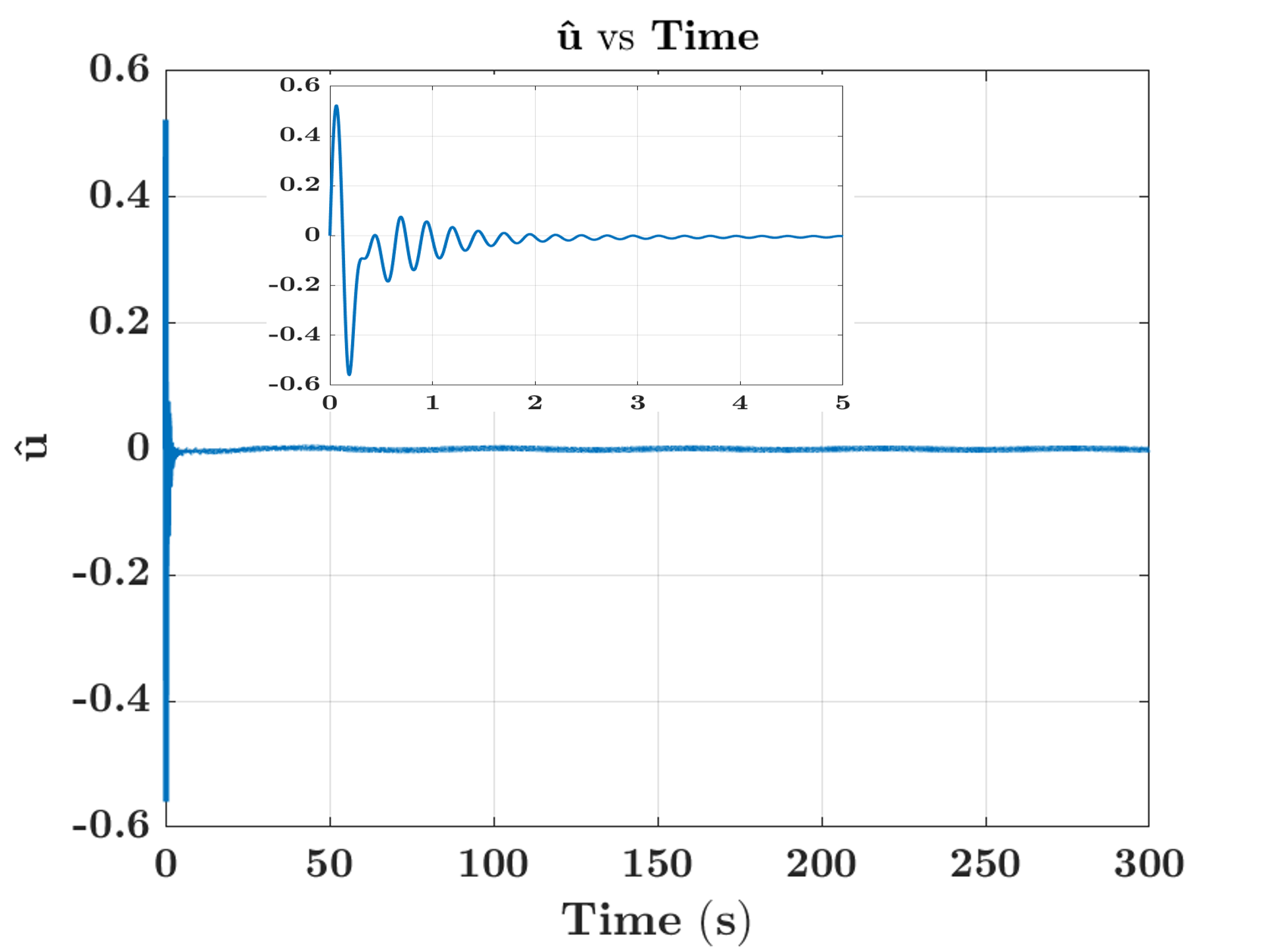}
    \caption{}
    \end{subfigure}\hspace{0.02\linewidth}
    \begin{subfigure}[b]{0.33\linewidth}
    \centering
    \includegraphics[width=\linewidth]{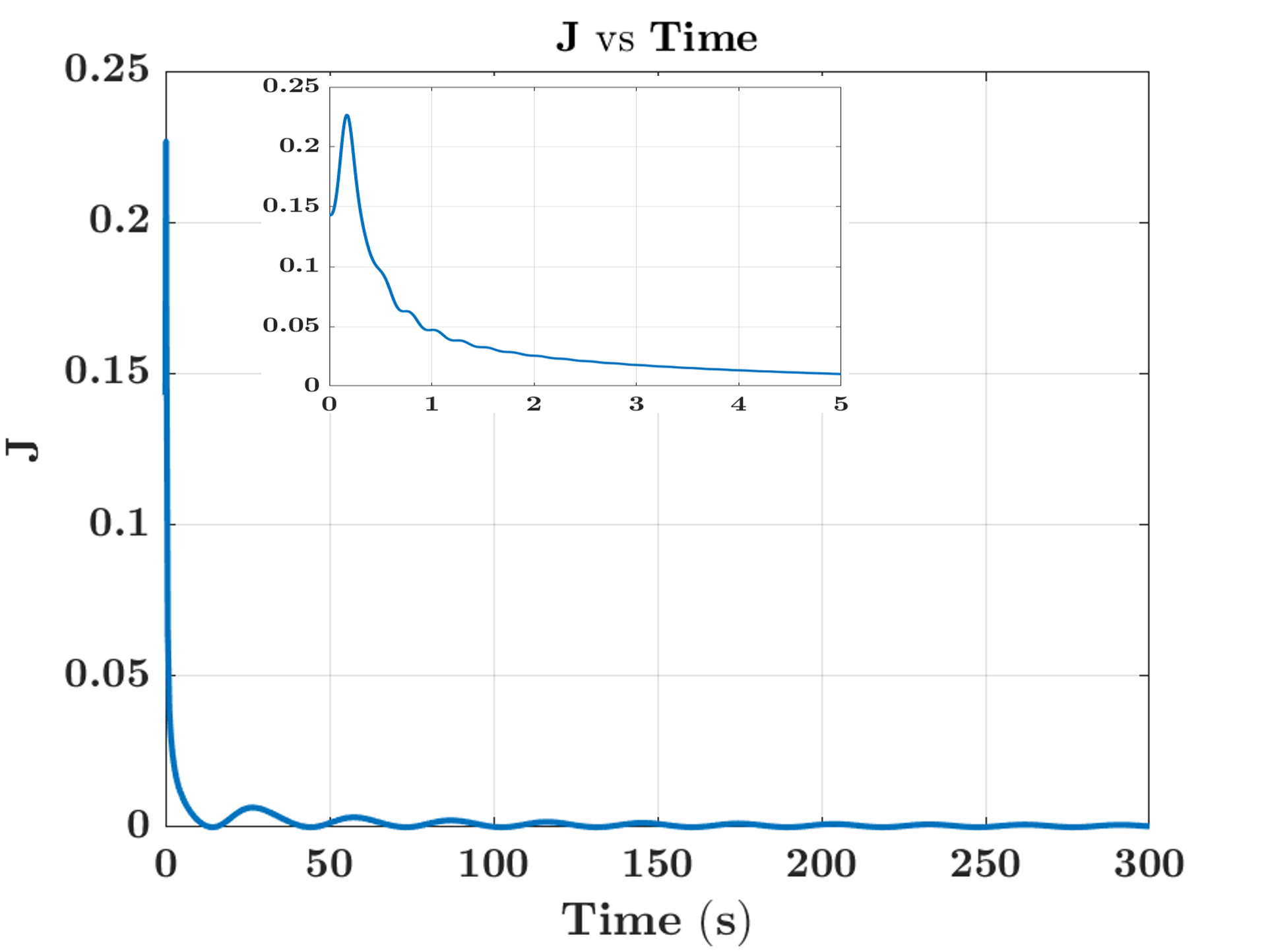}
    \caption{}
    \end{subfigure}
    \caption{Simulation Results for the quadcopter. (a) Euler angles, $\phi, \ \theta, \ \psi$, (b) Quadcopter angular velocity, $\bm{\Omega}$, (c) Control estimate, $\hat{u}$, (d) Objective function measurements, $J$.}
    \label{fig:quadcopter_simulation_results}
\end{figure*}

\begin{table}[h!]
\centering
\caption{ESC-VS Unicycle System Parameters and Initial Conditions}\label{tab:tabuni}
\begin{tabular}{@{}ll@{}}
\toprule
\textbf{Parameter} & \textbf{Value} \\ 
\midrule
$d_v,\ d_{\Omega}$ & $0.2,\ 0.1$ \\[3pt]
$x(0),\ y(0)$ & $2,\ 2\ \text{m}$ \\
$x_d,\ y_d$ & $1,\ 1\ \text{m}$ \\[3pt]
$v(0)$ & $0\ \text{m/s}$ \\
$\Omega(0)$ & $3\ \text{rad/s}$ \\
$\hat{u}(0)$ & $0$ \\[3pt]
$a_1,\ a_2$ & $4\times10^{-5},\ 5.2\times10^{-3}$ \\
$c_1,\ c_2$ & $0.4,\ 2.4$ \\
$k$ & $12.5$ \\
$\omega$ & $50$ \\
\bottomrule
\end{tabular}
\end{table}

\begin{figure*}[htbp]
    \centering
    \begin{subfigure}[b]{0.33\linewidth}
    \includegraphics[width=\linewidth]{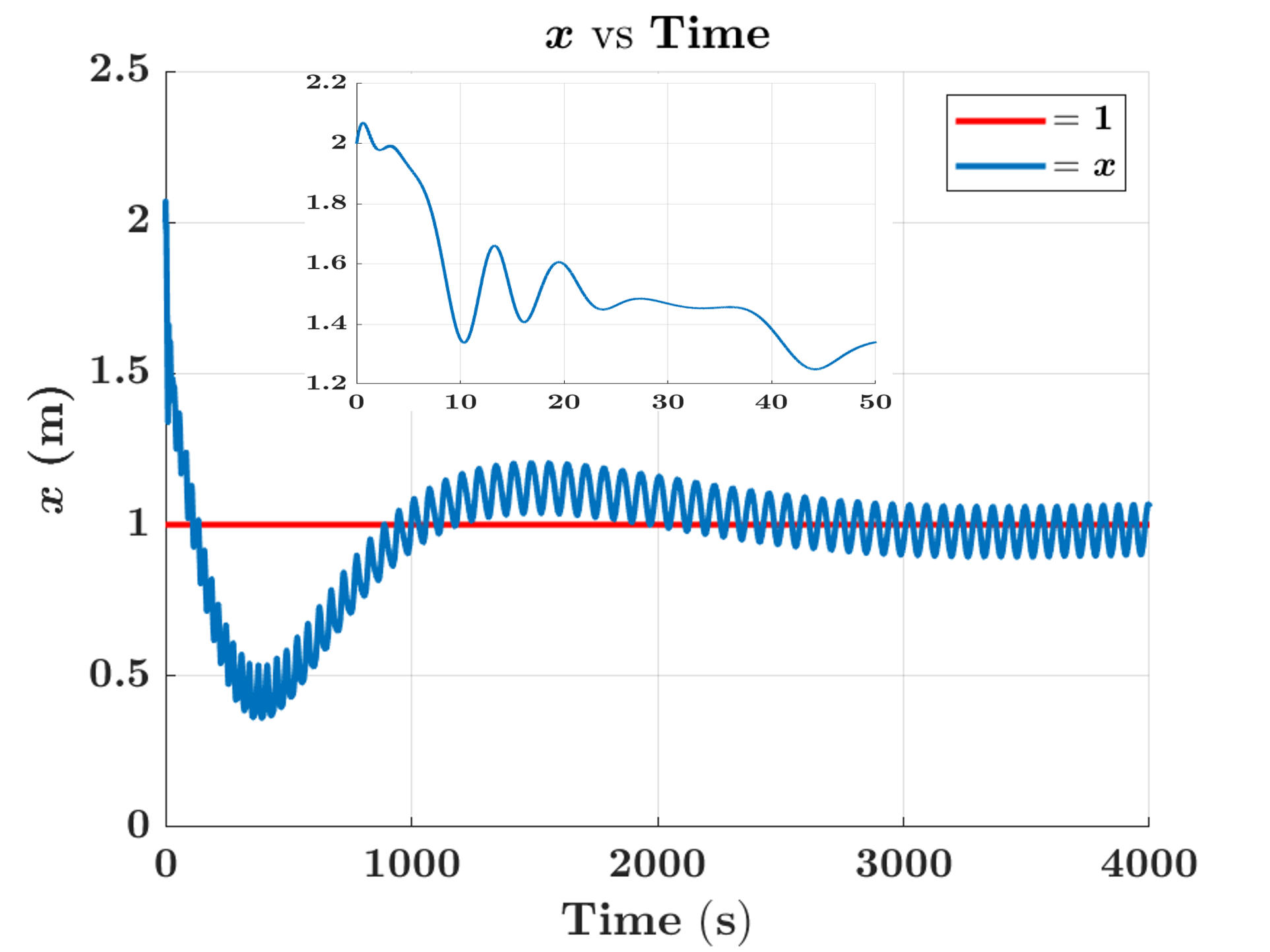}
    \caption{}
    \end{subfigure}\hfill
    \begin{subfigure}[b]{0.33\linewidth}
    \includegraphics[width=\linewidth]{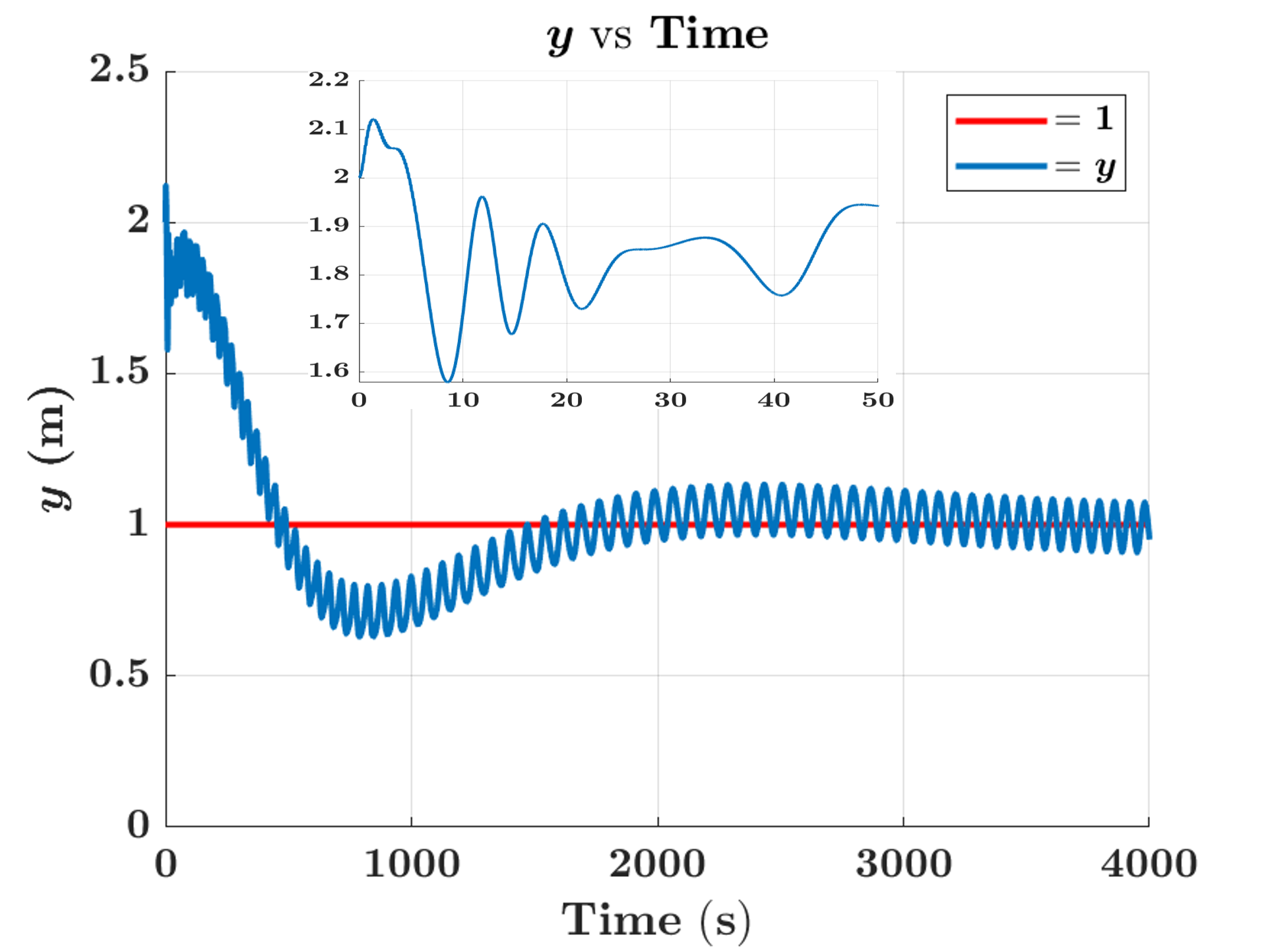}
    \caption{}
    \end{subfigure}\hfill
    \begin{subfigure}[b]{0.33\linewidth}
    \includegraphics[width=\linewidth]{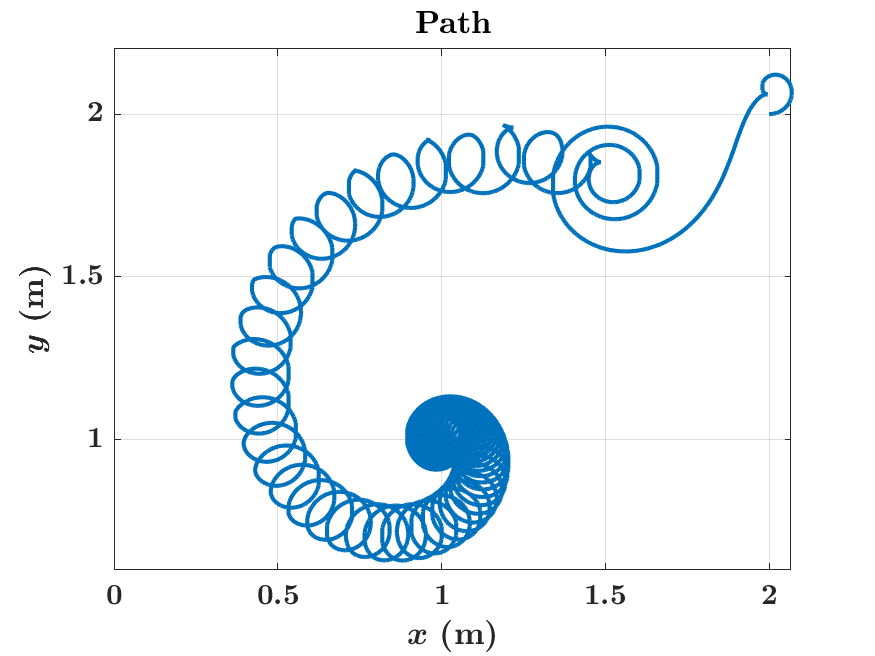}
    \caption{}
    \end{subfigure}\hfill
    \begin{subfigure}[b]{0.33\linewidth}
    \includegraphics[width=\linewidth]{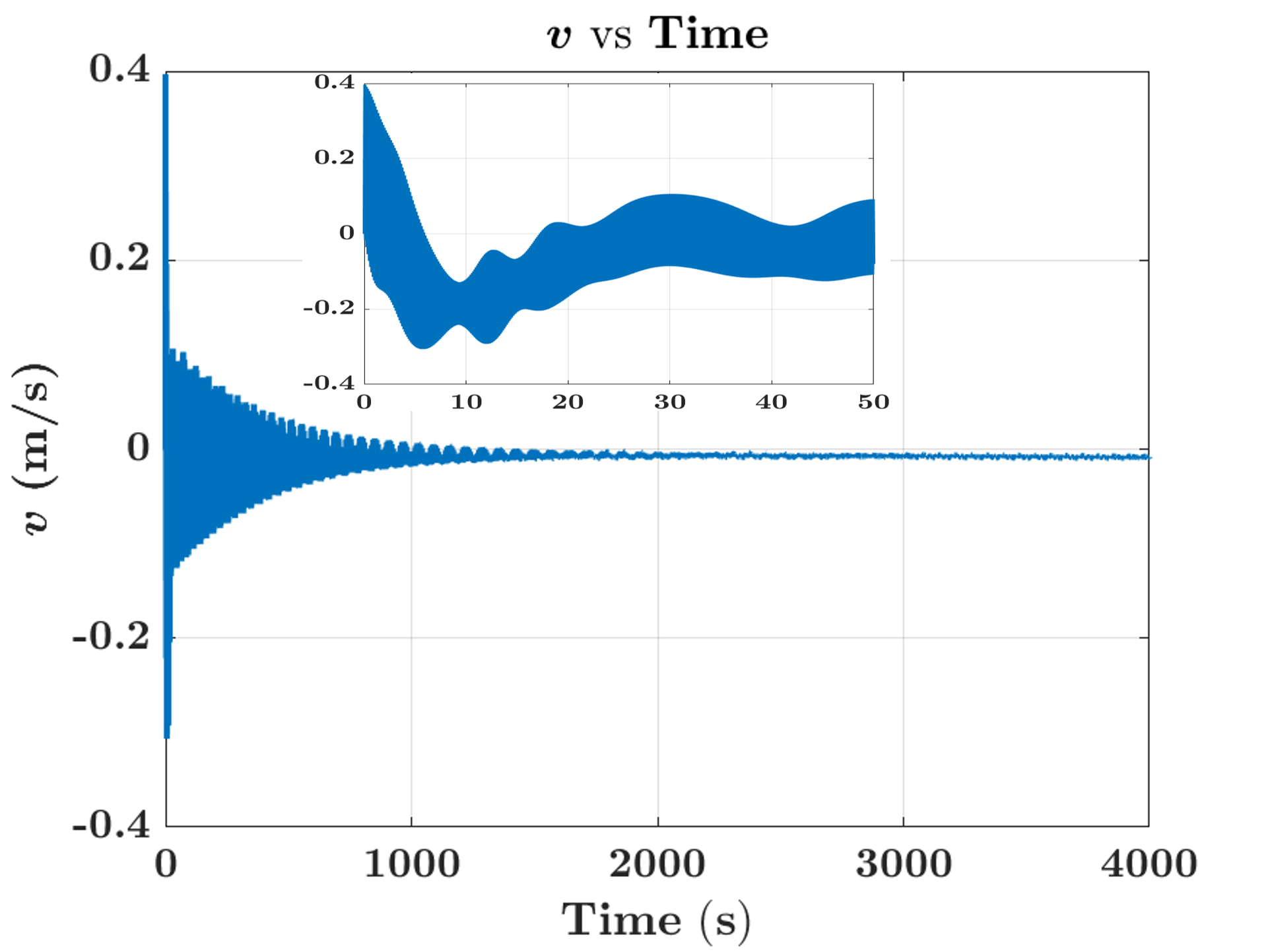}
    \caption{}
    \end{subfigure}\hfill
    \begin{subfigure}[b]{0.33\linewidth}
    \includegraphics[width=\linewidth]{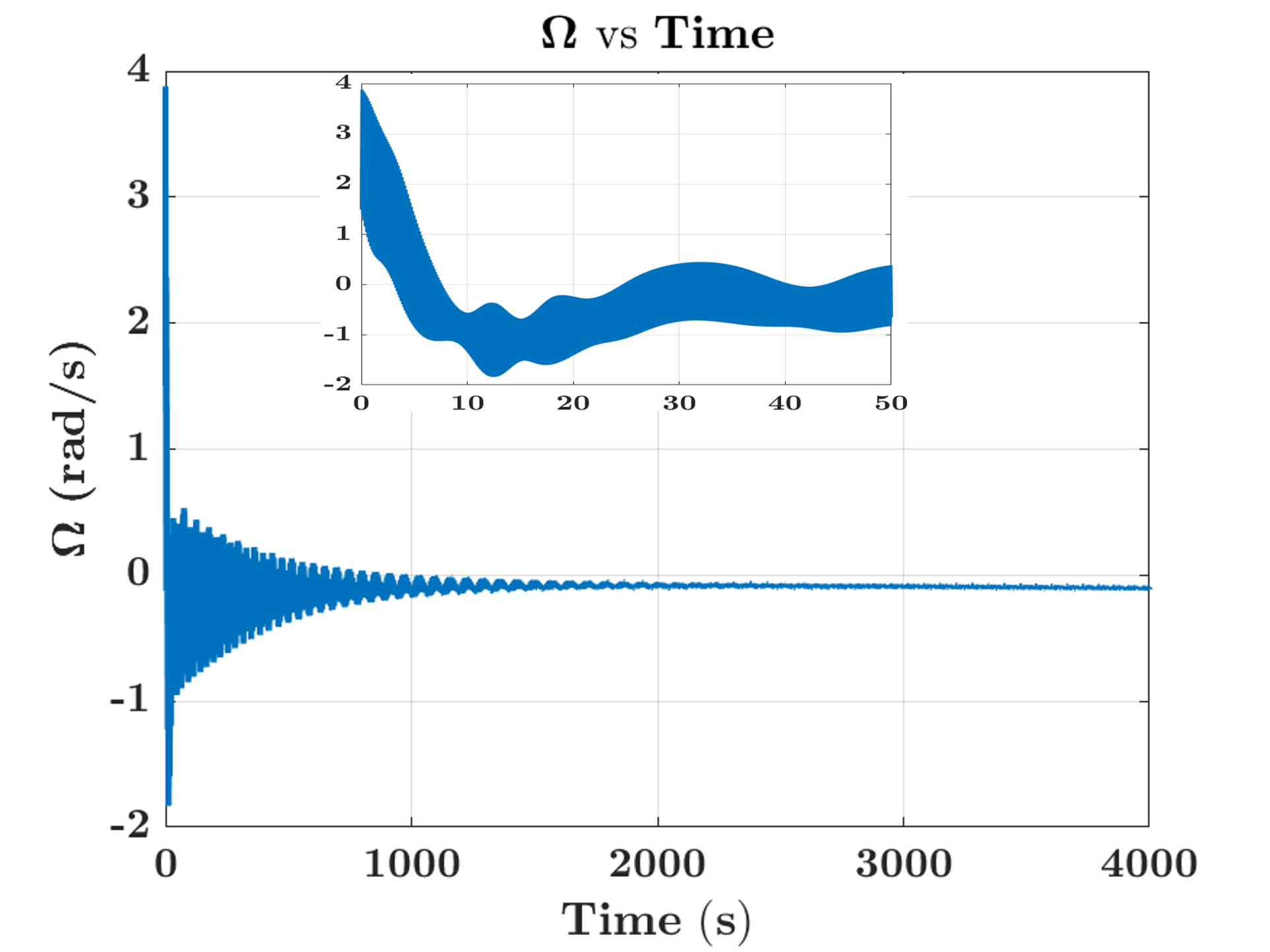}
    \caption{}
    \end{subfigure}\hfill
    \begin{subfigure}[b]{0.33\linewidth}
    \includegraphics[width=\linewidth]{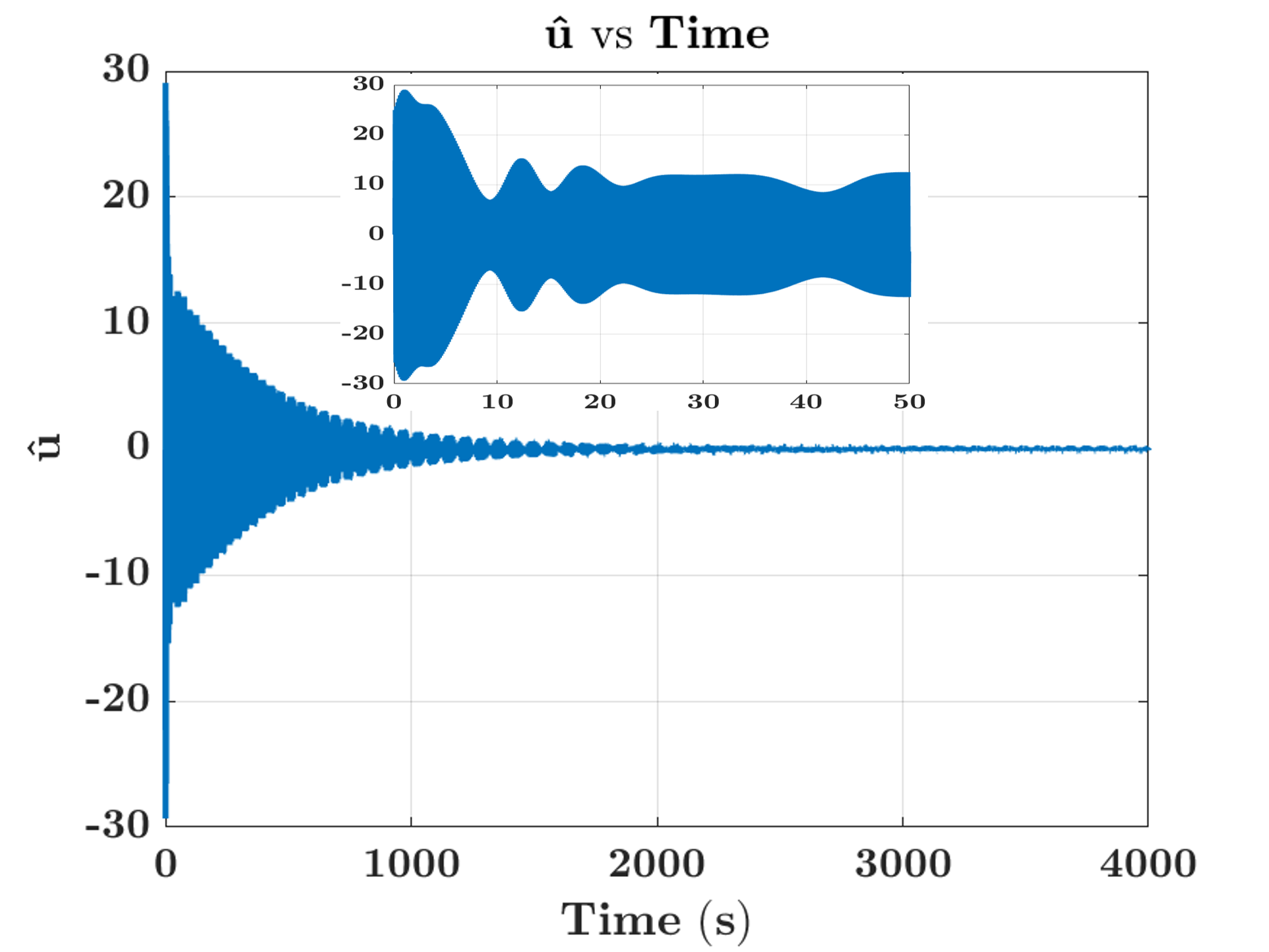}
    \caption{}
    \end{subfigure}\hfill
    \begin{subfigure}[b]{0.33\linewidth}
    \includegraphics[width=\linewidth]{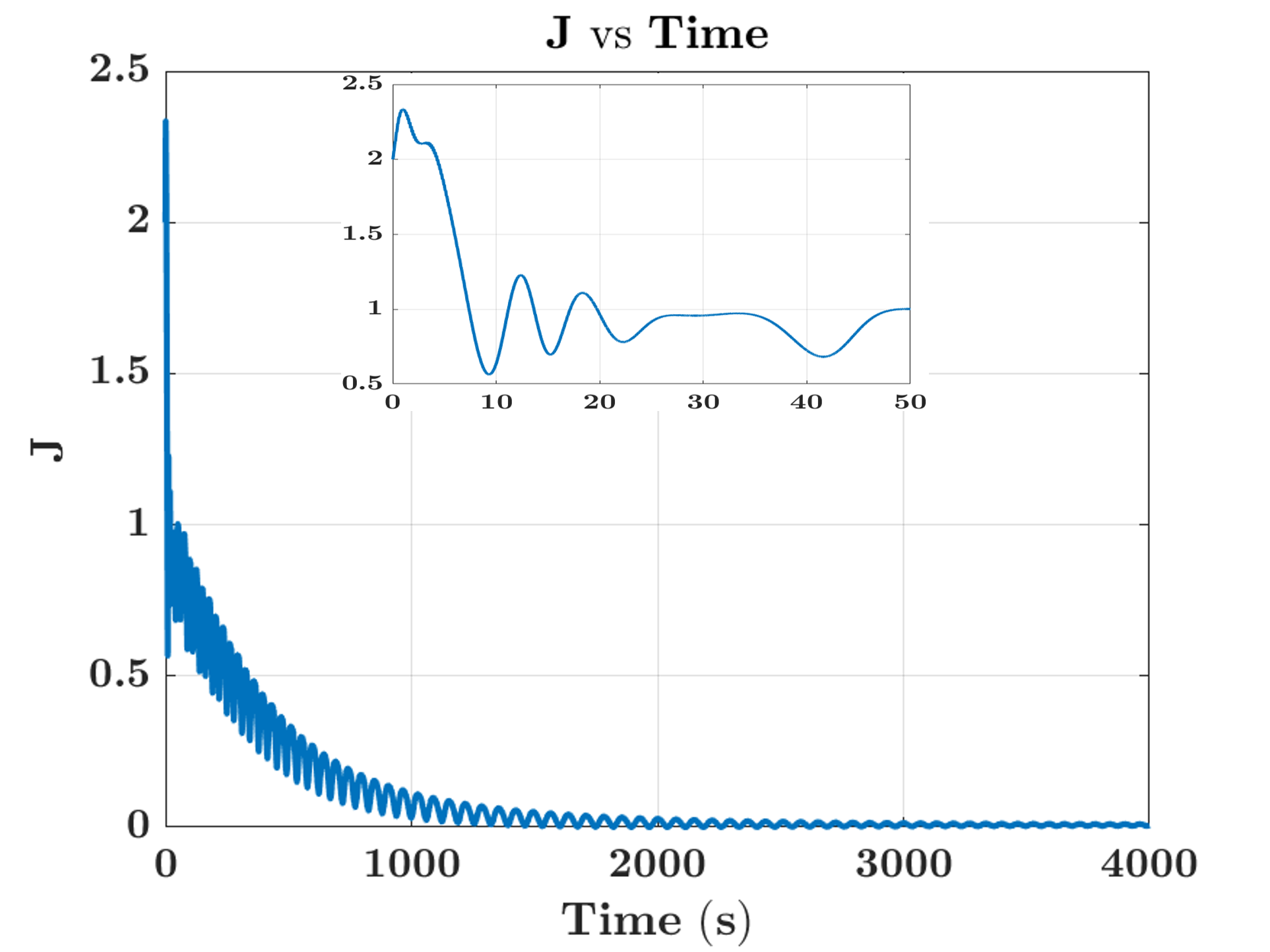}
    \caption{}
    \end{subfigure}
    \caption{Simulation Results for the unicycle model. (a) X position, $x$, (b) Y position, $y$, (c) Path of unicycle, (d) Unicycle velocity, $v$, (e) Unicycle angular velocity, $\Omega$, (f) Control estimate, $\hat{u}$, (g) Objective function, $J$.}
    \label{fig:unicycle_simulation_results}
\end{figure*}

\vspace{1cm}

\section{ESC-VS Response to Delay and Noise}\label{sec:robustness}
In this section, we investigate the response of ESC-VS control loop to output (measurement) delays and noise. This analysis, carried out on the ESC-VS controlled rigid body systems presented in Section \ref{sec:main}, is paramount to evaluate the robustness of the presented ESC-VS strategy to real-world scenarios. Numerical simulations are performed for all three platforms, for which we report the respective objective function response; however, for brevity, only the state response of one application is presented.
\begin{remark}
    To provide a fair evaluation of the ESC-VS system's response to noise and delay, the system and control parameters used throughout this section will be the same as those used in Section~\ref{sec:sim}. That is, the only differences will be in the parameters related to the delay or noise added to the system.
\end{remark}


\subsection{ESC-VS Response to Measurement Delay}\label{sec:delay}
The simulations presented in Section~\ref{sec:sim} are repeated by introducing fixed output delays of varying durations in the ESC-VS feedback loop. From the results shown in Fig.~\ref{fig:Delay}, it is observed that all ESC-VS-controlled systems under study tolerate output delays on the order of $10^{-2}$~s without noticeable loss of stability. As the delay magnitude increases, the convergence rate decreases and the size of the practical convergence neighborhood increases, which is consistent with the delayed feedback acting as an effective phase lag in the extremum-seeking adaptation loop. Furthermore, Fig.~\ref{fig:Delay_quad} shows that the effect of delay observed in the quadcopter objective function propagates directly to the attitude states. Similar behavior is observed for the state responses of the remaining applications and is omitted here for brevity. However, the codes and Simulink for all presented results can be found in \cite{elgohary2025github}.

\begin{figure*}[htbp]
    \centering
    \begin{subfigure}[b]{0.33\linewidth}
    \includegraphics[width=\linewidth]{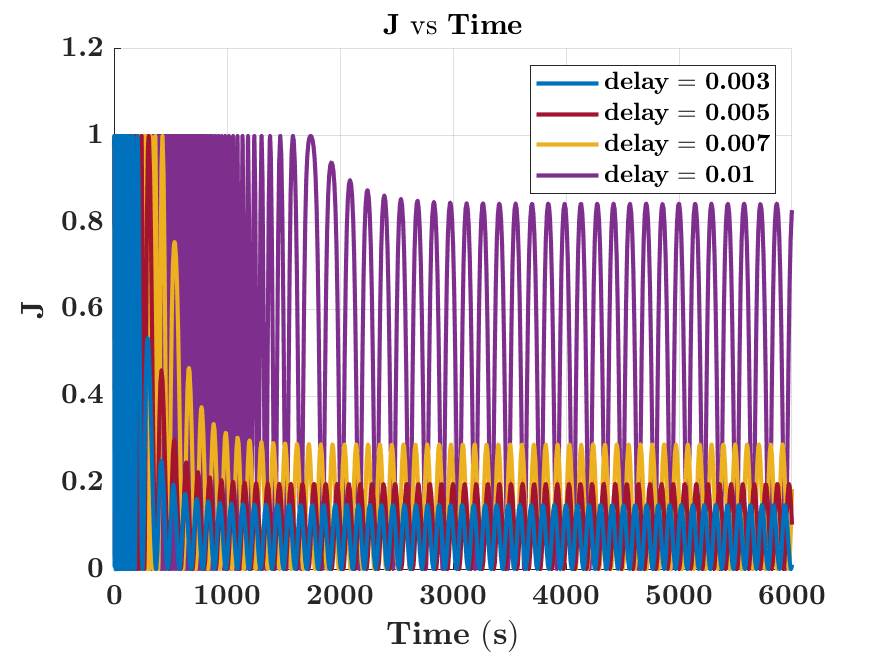}
    \caption{}
    \end{subfigure}\hfill
    \begin{subfigure}[b]{0.33\linewidth}
    \includegraphics[width=\linewidth]{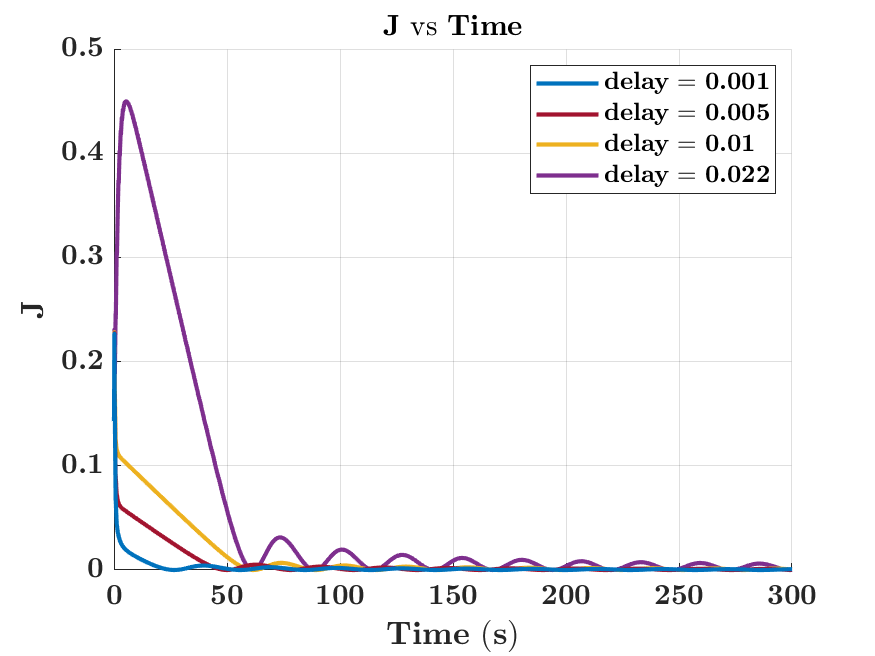}
    \caption{}
    \end{subfigure}\hfill
    \begin{subfigure}[b]{0.33\linewidth}
    \includegraphics[width=\linewidth]{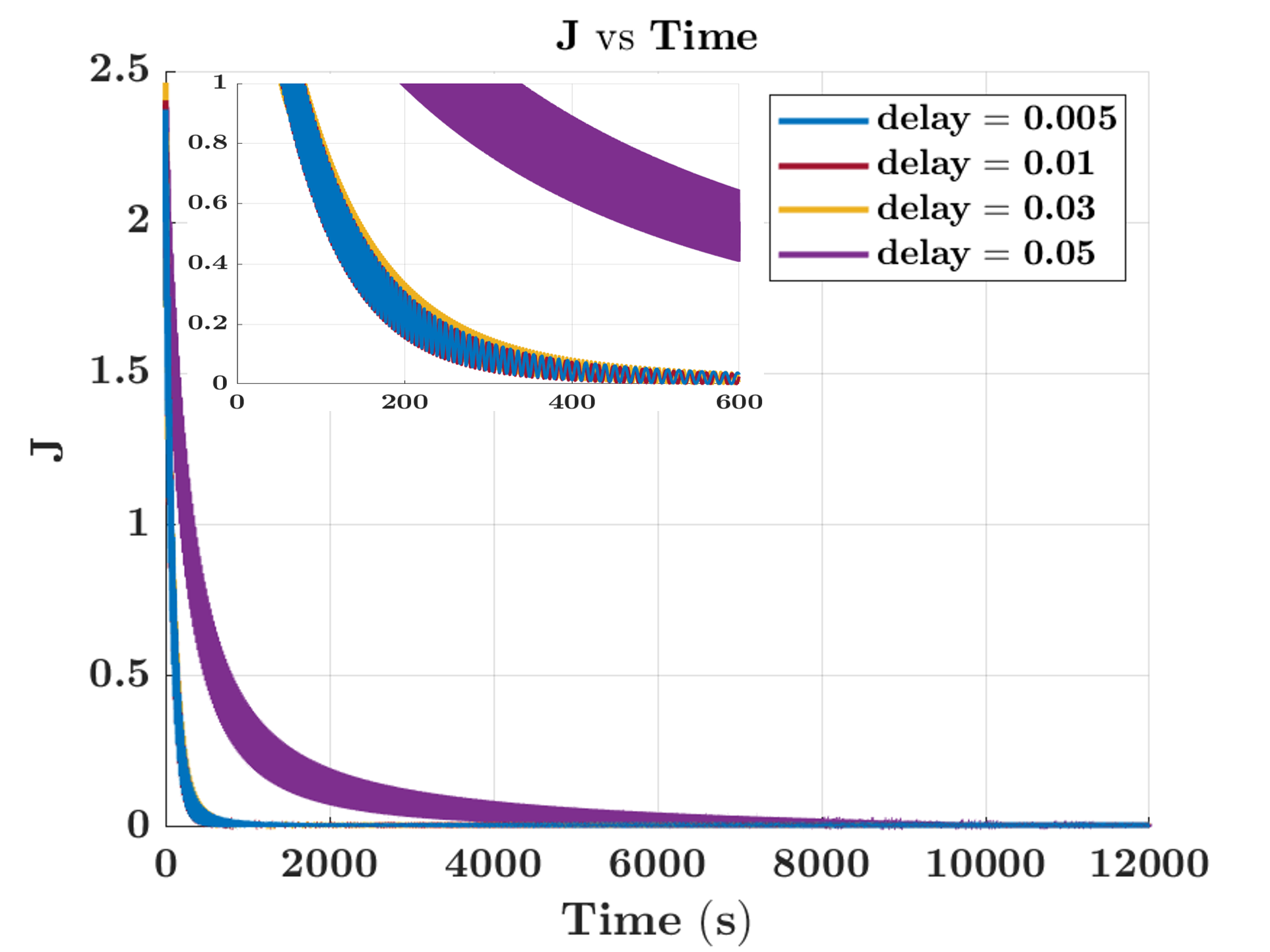}
    \caption{}
    \end{subfigure}
    \caption{Delay analysis. (a) Satellite objective function, (b) Quadcopter objective function, (c) Unicycle objective function.}
    \label{fig:Delay}
\end{figure*}
\begin{figure*}[htbp]
    \centering
    \begin{subfigure}[b]{0.33\linewidth}
    \includegraphics[width=\linewidth]{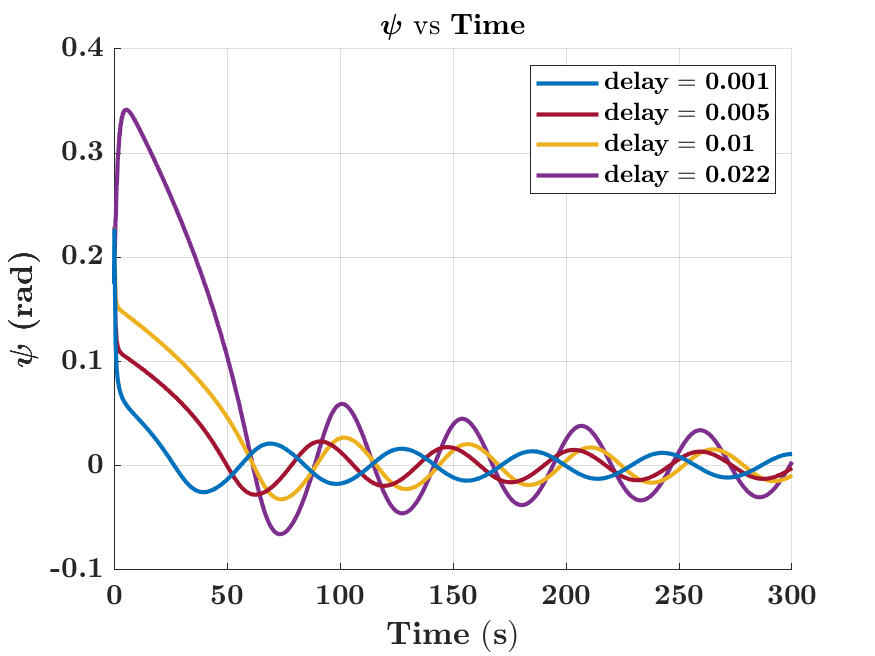}
    \caption{}
    \end{subfigure}\hfill
    \begin{subfigure}[b]{0.33\linewidth}
    \includegraphics[width=\linewidth]{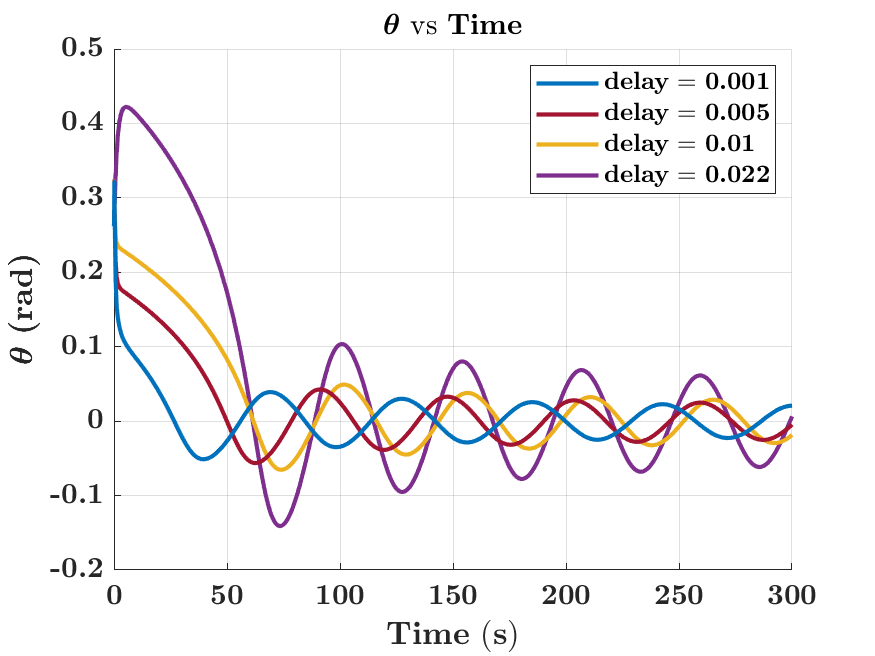}
    \caption{}
    \end{subfigure}\hfill
    \begin{subfigure}[b]{0.33\linewidth}
    \includegraphics[width=\linewidth]{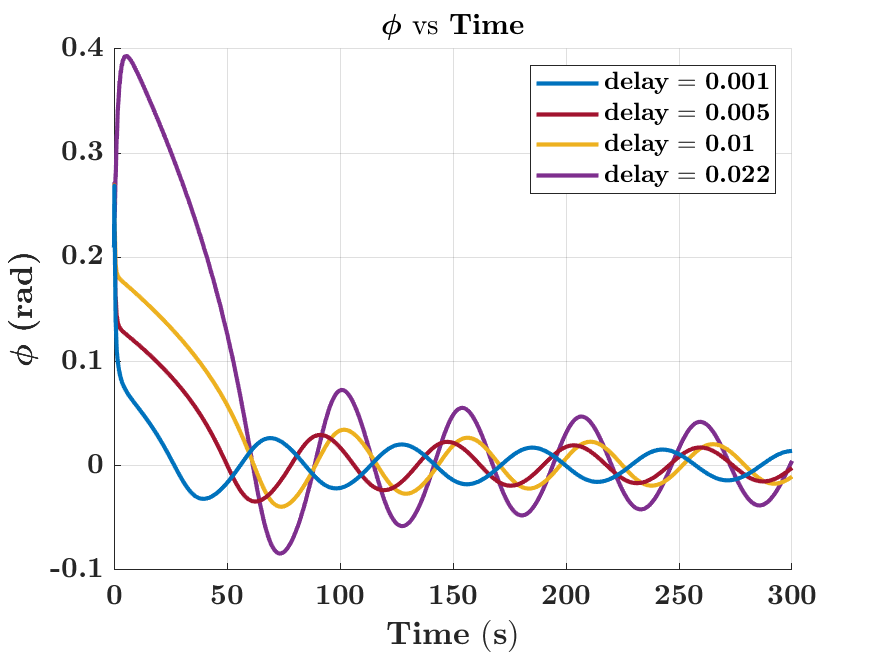}
    \caption{}
    \end{subfigure}
    \caption{Quadcopter state response to delays. (a) $\psi$, (b) $\theta$, (c) $\phi$.}
    \label{fig:Delay_quad}
\end{figure*}
Although ESC-VS displays encouraging results in handling delays, we note that, based on our trials, a control parameters tuning tailored to the presence of delays might improve the system performance. Nevertheless, a more directed and systematic approach would be to account for delays in the control design. Therefore, a future research direction will be developing ESC-VS with delay predictor similar to recent literature results such as \cite{oliveira2016extremum,doi:10.1137/1.9781611977356}. The results in \cite{oliveira2016extremum,doi:10.1137/1.9781611977356} are expected to be a good starting point to build upon for developing ESC-VS with delay predictor given their applicability to the classic structure of ESC systems, which present some structural similarities with ESC-VS in where the modulation and demodulation signals are applied but with different amplitude requirement.


\subsection{ESC-VS Response to Measurement Noise}\label{sec:noise}
Similar to the delay analysis in Section~\ref{sec:delay}, we repeat the simulations presented in Section~\ref{sec:sim} by injecting measurement noise into the objective function feedback. The noise is generated using the Simulink \emph{Random Number} block with uniform distribution, zero mean, and varying variance levels. A fixed random seed is used to ensure reproducibility across simulations.

Figure~\ref{fig:Noise} reports the objective function response for the satellite, quadcopter, and unicycle systems under different noise intensities. The results show that the ESC-VS framework maintains stability and bounded convergence behavior in the presence of measurement noise for all three platforms. As expected for extremum seeking–based methods, increasing noise variance leads to larger steady-state oscillations and a degradation in convergence accuracy; however, the systems remain practically stable and continue to operate in real time, tolerating non-negligible noise levels that causes at times over 25\% error in the objective function measurements.

To further illustrate the effect of noise on the system states, Figure~\ref{fig:Noise_sat} presents the satellite attitude response under noisy objective function measurements. The state trajectories exhibit bounded oscillations consistent with the objective function behavior, confirming that the noise-induced variations propagate coherently through the ESC-VS feedback loop. Similar trends were observed for the quadcopter and unicycle state responses and are omitted here for brevity. However, the codes and Simulink for all presented results can be found in \cite{elgohary2025github}

Overall, these results demonstrate that ESC-VS exhibits a meaningful degree of robustness to measurement noise, which is consistent with known properties of extremum seeking control schemes. While noise affects convergence accuracy, the method preserves stability and practical performance, motivating future work on enhanced filtering strategies and noise-aware ESC-VS designs.

\begin{figure*}[ht]
    \centering
    \begin{subfigure}[b]{0.33\linewidth}
    \includegraphics[width=\linewidth]{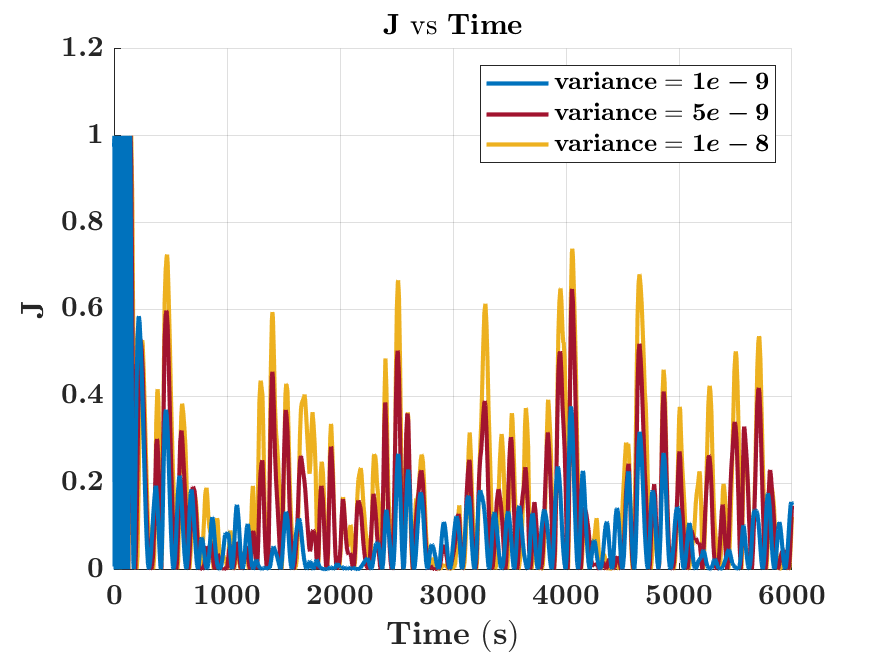}
    \caption{}
    \end{subfigure}\hfill
    \begin{subfigure}[b]{0.33\linewidth}
    \includegraphics[width=\linewidth]{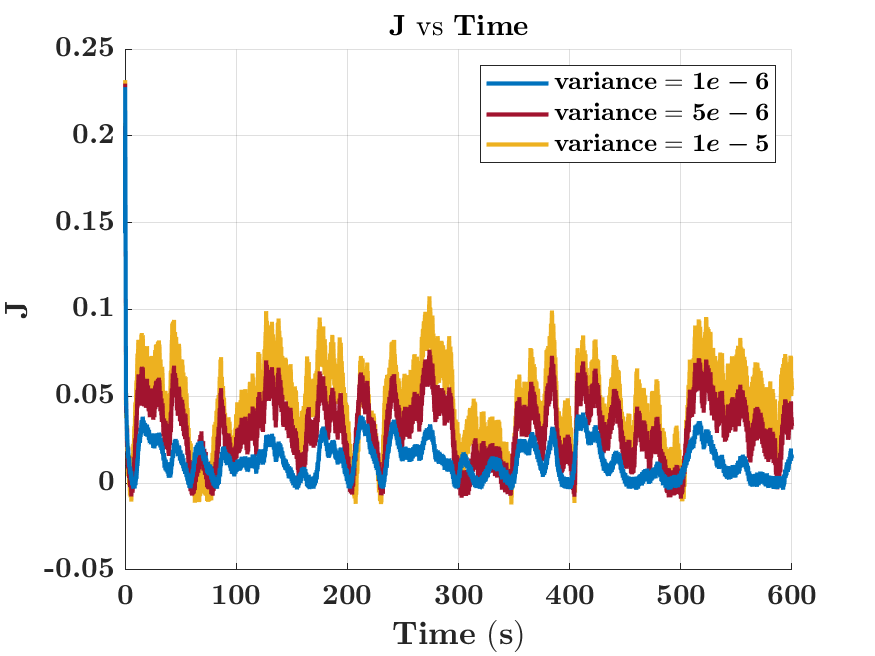}
    \caption{}
    \end{subfigure}\hfill
    \begin{subfigure}[b]{0.33\linewidth}
    \includegraphics[width=\linewidth]{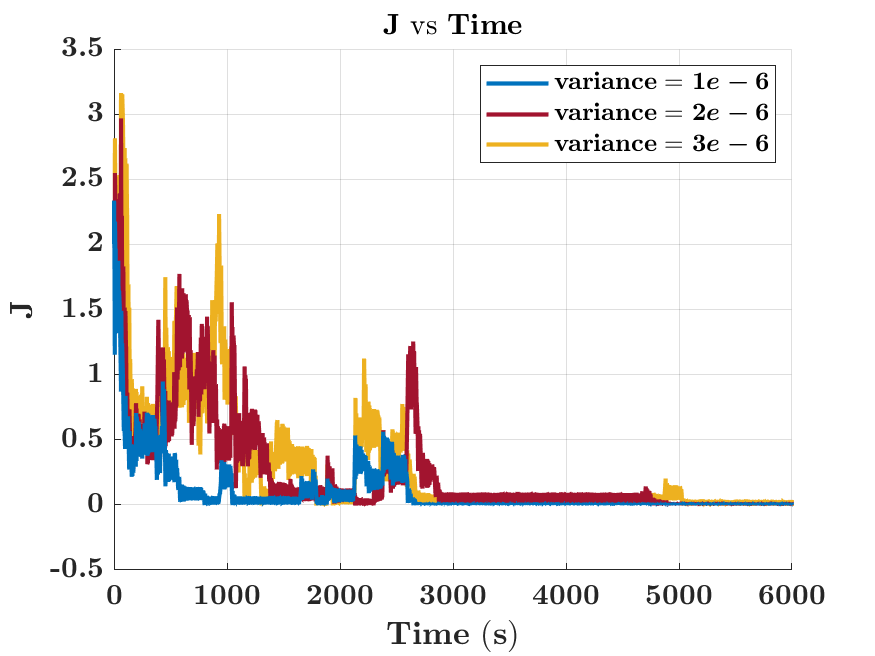}
    \caption{}
    \end{subfigure}
    \caption{Noise analysis. (a) Satellite objective function, (b) Quadcopter objective function, (c) Unicycle objective function.}
    \label{fig:Noise}
\end{figure*}
\begin{figure*}[ht]
    \centering
    \begin{subfigure}[b]{0.33\linewidth}
    \includegraphics[width=\linewidth]{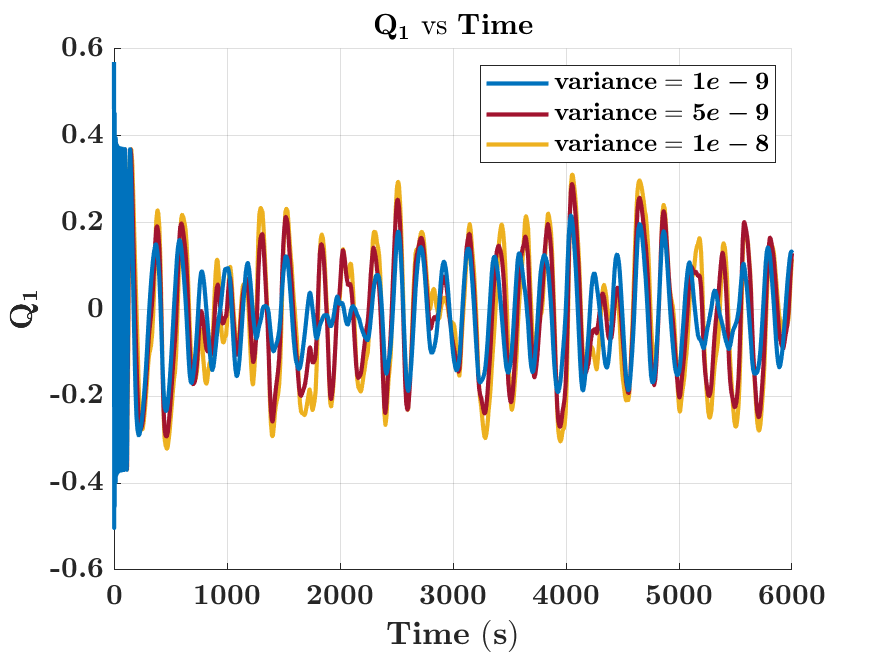}
    \caption{}
    \end{subfigure}\hfill
    \begin{subfigure}[b]{0.33\linewidth}
    \includegraphics[width=\linewidth]{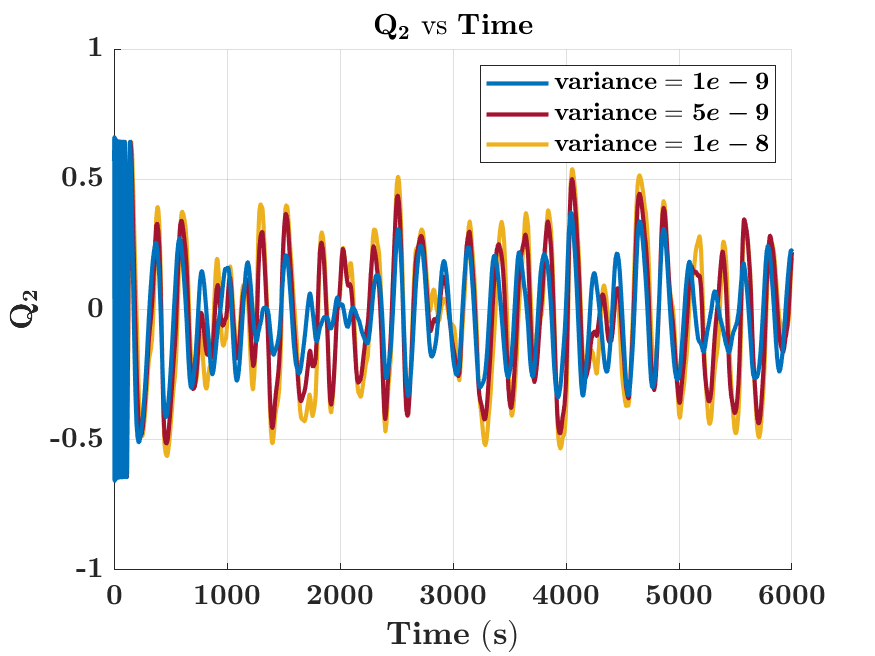}
    \caption{}
    \end{subfigure}\hfill
    \begin{subfigure}[b]{0.33\linewidth}
    \includegraphics[width=\linewidth]{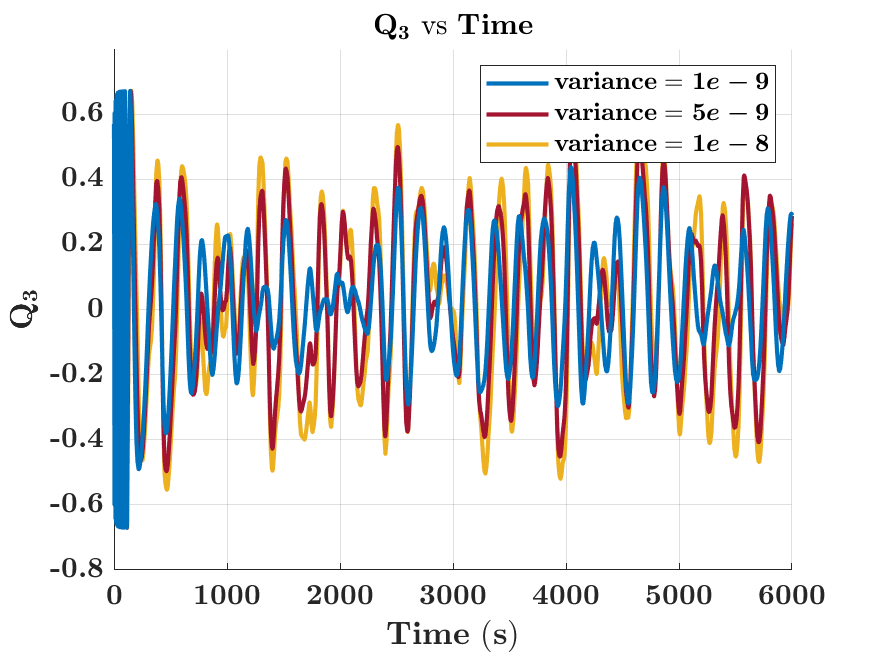}
    \caption{}
    \end{subfigure}
    \caption{Satellite state response to Noise. (a) $\textbf{Q}_1$, (b) $\textbf{Q}_2$, (c) $\textbf{Q}_3$.}
    \label{fig:Noise_sat}
\end{figure*}
\section{Conclusion}
In this paper, we succeeded in introducing the recent approach of extremum seeking for vibrational stabilization (ESC-VS) in \cite{elgohary2025letters} to a class of rigid body dynamics as depicted in Figure \ref{fig:esc_diagram}. We characterized the ESC-VS rigid body system and its stability. Moreover, we demonstrated its applicability to rigid body dynamical systems, including with measurement delays and noise, via three important applications: (1) satellite attitude dynamics, (2) quadcopter attitude dynamics, and (3) acceleration-controlled unicycle dynamics. We implemented case studies using numerical simulations for said three applications; results illustrate the effectiveness of the proposed ESC-VS rigid body system, which is able to operate in model-free, real-time fashion and using only one perturbation/vibrational signal for stabilization unlike literature results on extremum seeking approaches for rigid body dynamics. 

In the future, we aim at expanding the current ESC-VS method to broader classes of rigid body and mechanical systems, as well as more applications. Additionally, delay predictor/compensator and further filters for reversing noise effects can be investigated, including expanding the variation of averaging technique used to characterize the stability of ESC-VS to enable delays and noise.    

\bibliography{ifacconf}             

\end{document}

%% file: PSmacros.tex
 \usepackage{etex}
 
\usepackage{lipsum}
\usepackage{amsfonts}
\usepackage{graphicx}
\usepackage{epstopdf} 
\usepackage{enumerate} 
 \usepackage{amsmath}  
 \usepackage{subfig,bbm,enumerate,multirow,framed,blkarray}

\usepackage{verbatim,tcolorbox}
  \usepackage{tikz}
\usetikzlibrary{shapes,arrows}
\usetikzlibrary{positioning}  
\usepackage{pstricks-add} 
 \usepackage{wrapfig}
 \usepackage{pgfplots}
  
\usepackage{color}
\usepackage{savesym}
 \usepackage{algpseudocode} 

 \usepackage{graphicx}          

\newtheorem{theorem}{Theorem}[section] 
\newtheorem{corollary}[theorem]{Corollary}
\newtheorem{lemma}[theorem]{Lemma}

\newtheorem{remark}[theorem]{Remark}

	\algnewcommand\And{\textbf{and }}
	\algnewcommand\Or{\textbf{or }}

\algnewcommand{\algorithmicgoto}{\textbf{go to}}%
\algnewcommand{\Goto}[1]{\algorithmicgoto~\ref{#1}}%